\documentclass[a4paper]{amsart}
\calclayout

\usepackage[T1]{fontenc}
\usepackage{lmodern}

\usepackage{enumerate}

\usepackage{mathtools}
\usepackage{amssymb,amsmath}

\usepackage[sort&compress,numbers]{natbib}

\theoremstyle{definition}
	\newtheorem{defn}{Definition} 
\theoremstyle{plain}
	\newtheorem{thm}[defn]{Theorem}
	\newtheorem{lem}[defn]{Lemma}
	\newtheorem{cor}[defn]{Corollary}
	\newtheorem{prop}[defn]{Proposition}
\theoremstyle{remark}

\newcommand{\N}{\mathbb N}

\newcommand{\R}{\mathbb R}

\newcommand{\dd}{\mathrm{d}}
\newcommand{\PP}{\mathbb P}
\newcommand{\EE}{\mathbb E}
\DeclareMathOperator{\Var}{Var}
\DeclareMathOperator{\sgn}{sgn}
\DeclareMathOperator{\zeros}{zeros}
\newcommand{\ind}{\mathbf{1}}

\newcommand{\Fc}{\mathcal F}
\newcommand{\Gc}{\mathcal G}
\newcommand{\Sc}{\mathcal S}
\newcommand{\Ac}{\mathcal A}

\newcommand{\Xt}{\widetilde{X}}
\newcommand{\at}{\tilde{a}}
\newcommand{\bt}{\tilde{b}}
\newcommand{\At}{\widetilde{A}}
\newcommand{\Bt}{\widetilde{B}}

\newcommand{\Ab}{\bar{A}}
\newcommand{\Bb}{\bar{B}}
\newcommand{\Xb}{\bar{X}}
\newcommand{\ah}{\hat{a}}
\newcommand{\bh}{\hat{b}}
\newcommand{\Xhat}{\hat{X}}

\newcommand{\Xhwpt}{\widehat{X}^{\mathrm{WPt}}} 
\newcommand{\Xhe}{\widehat{X}^{\mathrm{E}}} 
\newcommand{\Xaux}{\overline{X}^{\operatorname{aux}}}
\newcommand{\Xh}{\widehat X}
\newcommand{\Xm}{\widehat{X}^{\mathrm{M}}} 
\newcommand{\Xmaux}{\overline{X}^{\mathrm{M}}} 

\newcommand{\zsf}{{\mathcal G}}
\newcommand{\msf}{{\mathcal M}}
\newcommand{\Ups}{{\mathcal Y}}
\newcommand{\Upsh}{\widehat{{\mathcal Y}}}

\newcommand{\Mh}{\widehat{m}}

\newcommand{\Qb}{\overline{Q}}
\newcommand{\msfh}{\widehat{\msf}}

\newcommand{\Rh}{\widehat{R}}
\newcommand{\A}{{\mathfrak{A}}}
\newcommand{\F}{{\mathfrak{F}}}

\newcommand{\Aa}{\mathbb A}

\begin{document}

\title[Lower Error Bounds for Strong Approximation of Scalar SDEs]
	{Lower Error Bounds for Strong Approximation of Scalar SDEs with non-Lipschitzian Coefficients}

\author[Hefter]{Mario Hefter}
\address{Mario Hefter\\
		Johann Radon Institute for Computational and Applied Mathematics\\
		Austrian Academy of Sciences\\
		Altenbergerstra{\ss}e 69\\
		4040 Linz\\
		Austria}
\email{mario.hefter@ricam.oeaw.ac.at}

\author[Herzwurm]{Andr\'{e} Herzwurm}
\address{Andr\'{e} Herzwurm\\
	Fachbereich Mathematik\\
	Technische Universit\"at Kaisers\-lautern\\
	Postfach 3049\\
	67653 Kaiserslautern\\
	Germany}
\email{herzwurm@mathematik.uni-kl.de}

\author[M\"uller-Gronbach]{Thomas M\"uller-Gronbach}
\address{Thomas M\"uller-Gronbach\\
	Fakult\"at f\"ur Informatik und Mathematik\\
	Universit\"at Passau\\
	Innstra{\ss}e 33\\
	94032 Passau\\
	Germany}
\email{thomas.mueller-gronbach@uni-passau.de}

\begin{abstract}
	We study pathwise approximation of scalar stochastic differential equations
	at a single time point or globally in time by means of methods that are based on
	finitely many observations of the driving Brownian motion. We prove lower 
	error bounds in terms of the average number of evaluations of the driving Brownian motion
	that hold for every such method under rather mild assumptions on
	the coefficients of the equation. The underlying simple idea of our analysis is as follows:       
	the lower error bounds known for equations with coefficients that have 
	sufficient regularity globally in space should still apply in the case of coefficients that 
	have this regularity in space only locally, in a small neighborhood of the initial value. 
	Our results apply to a huge variety of equations with coefficients that are not globally 
	Lipschitz continuous in space 
	including Cox-Ingersoll-Ross processes, equations with superlinearly growing coefficients, and 
	equations with discontinuous coefficients. In many of these cases the resulting lower error bounds even turn out to be sharp.  
\end{abstract}

 
\maketitle

\setcounter{tocdepth}{4}\setcounter{secnumdepth}{4}
\tableofcontents

\section{Introduction}\label{sec:intro}

Let $T>0$ and consider a scalar stochastic differential equation (SDE)
\begin{equation}\label{sde1}
		\dd X(t) = a(t,X(t)) \,\dd t + b(t,X(t)) \,\dd W(t), \qquad t\in [0,T],
\end{equation}
with drift coefficient $a\colon [0,T]\times \R \to\R$, diffusion coefficient $b\colon [0,T]\times \R \to\R$,
one-dimensional driving Brownian motion $W$, and initial value $X(0)$
such that \eqref{sde1} has a solution $X=(X(t))_{t\in[0, T]}$.
The computational problem we study is strong approximation of the solution $X$,
either globally on the whole time interval $[0,T]$ or at the final time $T$,
by means of methods that may use the initial value $X(0)$ and a finite number of sequentially taken evaluations $W(\tau_1),\dots,W(\tau_\nu)$ of the driving Brownian motion $W$ at times $\tau_1,\dots,\tau_\nu\in [0,T]$. 
Except for  measurability conditions we do not impose any further restrictions. The
$k$-th site $\tau_k$ may depend on the previous evaluations $X(0),W(\tau_1),\dots,W(\tau_{k-1})$, e.g., by using a path-dependent step size control,
and the total number $\nu$ of observations of $W$
may be determined by a stopping rule. Finally, the resulting discrete data may
be used in any way to generate an approximation to $X$ or to $X(T)$.
See Section~\ref{sec:methods} for a formal description of such approximations. 
Our goal is to establish lower error bounds that hold for any such method, in terms of the average number $\EE[\nu]$
of evaluations of $W$ that are used.

Lower error bounds for strong approximation of (systems of) SDEs based on evaluations of the
driving Brownian motion at finitely many times were first established in 1980 by \citet*{clark:1980}
in the particular case of strong approximation of Lev\'{y} areas.
Meanwhile, lower error bounds have extensively been studied in the case of coefficients
that are globally Lipschitz continuous in the state variable and sufficiently smooth,
see
\citet*{Ruemelin1982,CambanisHu1996,HMGR00-1,HMGR00-2,HMGR01,MG02_habil,MG02,MG04}.
Moreover, under the assumption of global Lipschitz continuity in the state variable,
lower error bound results are also available for equations with coefficients that are discontinuous in time,
see \citet*{Przybylowicz2014,Przybylowicz2015a,Przybylowicz2015b},
for stochastic delay differential equations, see \citet*{Hofmann-MuellerGronbach2006},
and for equations driven by a fractional Brownian motion, see \citet*{neuenkirch06,neuenkirch08,neuenkirch-shalaiko16}.

For SDEs with coefficients that are not globally Lipschitz continuous in the state variable
investigations on lower error bounds have started only recently. There seem to be two directions of research, up to now.
One of them consists in establishing sub-polynomial lower error bounds for particular equations
with smooth coefficients in order to come closer to a characterization of polynomial convergence  in that case,
see \citet*{hhj15,Jentzen-MG-Yaroslavtseva2016,yaroslavtseva2017,Gerencser-Jentzen-Salimova2017,MG-Yaroslavtseva2017}.
The other one aims at a thorough analysis of strong approximation of Cox-Ingersoll-Ross processes
as a prototype of  SDEs with a diffusion coefficient that is H\"older continuous in the state variable
with a H\"older exponent strictly between zero and one, see \citet*{HH16a,HH16b,HJ17}.

In the present paper we aim at scalar equations~\eqref{sde1} with coefficients $a$ and $b$
that are not globally Lipschitz continuous in the state variable and we establish lower error bounds
under rather mild assumptions on $a$ and $b$ by exploiting, essentially, the following simple idea:
it is likely that the lower error bounds known for equations with coefficients that have 
sufficient regularity globally in space still apply in the case of coefficients that 
have this regularity in space only locally, in a small neighborhood of the initial value. 

To give a flavour of our results we consider for simplicity the particular case
of an autonomous equation~\eqref{sde1}, i.e., we assume that
\begin{itemize}
	\item[(A)] $(\Omega,\F,\PP)$ is a complete probability space with a normal filtration
	$(\Fc_t)_{t\in[0,T]}$,  $W\colon [0,T]\times\Omega\to\R$
	is a standard Brownian motion on $(\Omega,\F,\PP)$ with respect to
	$(\Fc_t)_{t\in [0,T]}$,  $a\colon \R\to\R$
	and $b\colon \R\to\R$ are Borel-measurable functions, and
	$X\colon [0,T]\times\Omega\to\R$ is an $(\Fc_t)_{t\in [0,T]}$-adapted
	stochastic process with continuous sample paths such that $\PP$-a.s.~
	$\int_0^T (|a(X(t))| + |b(X(t))|^2)\,\dd t < \infty$ and for all $t\in[0,T]$ $\PP$-a.s.~
	\begin{align*}
		X(t) = X(0) + \int_0^t a(X(s))\,\dd s + \int_0^t b(X(s))\,\dd W(s),
	\end{align*}
\end{itemize}
and we restrict ourselves to approximations of the solution $X$ that are based on finitely many evaluations
of the driving Brownian motion at fixed times in $[0,T]$. Note, however, that all of the following lower error bounds also hold for approximations that use $n$ sequentially taken evaluations of the driving Brownian motion $W$ on average, see Sections~\ref{sec:ft} and~\ref{sec:global}.

We first consider strong approximation of the solution globally on $[0,T]$ with respect to the supremum-norm. The following result is an immediate consequence of Theorem~\ref{thm:sup-local}
in Section~\ref{sec:global-sup}.

\begin{thm}[$L_\infty$-approximation]\label{thm:intro-sup}
	Assume $(A)$. Let $t_0\in[0,T)$ and let $\emptyset\neq I\subseteq\R$ be an open interval such that
	\begin{enumerate}[\hspace{.5cm}(i)]\addtolength{\itemsep}{0.25\baselineskip}
		\item $a,b$ are once continuously differentiable on $I$,
		\item $\forall\,x\in I\colon\, b(x)\neq 0$,
		\item $\PP(X(t_0)\in I)>0$.
	\end{enumerate}
	Then there exist constants $c,\gamma\in(0,\infty)$ such that for all $n\in\N$, for all $s_1,\dots,s_n\in [0,T]$ 
	and for all measurable mappings $u\colon\R^{n+1}\to C([0,T])$ we have
	\begin{align*}
		\PP\bigl( \|X-u(X(0), W( s_1 ), \dots, W( s_n ) )\|_\infty\geq c\cdot \sqrt{\ln(n+1)/n} \bigr) \geq \gamma.
	\end{align*}
	In particular, we have for all $n\in\N$ that
	\begin{align*}
		\inf_{\substack{ s_1,\dots,s_n\in [0,T]\\ u\colon \R^{n+1}\to C([0,T]) \text{ measurable} }}\hspace{-0.5cm}
			\EE\bigl[ \|X - u(X(0), W( s_1 ), \dots, W( s_n ) ) \|_\infty \bigr]
			\geq c\gamma \cdot \sqrt{\ln(n+1)/n}.
	\end{align*}
\end{thm}

Under slightly stronger smoothness assumptions on the coefficients $a$ and $b$
we obtain lower bounds for the error with respect to the $L_1([0,T])$-norm (denoted by $\|\cdot\|_1$),
which in turn implies a lower bound for the maximum pointwise approximation error.
The following result is an immediate consequence of Theorem~\ref{thm:Lp-local} in Section~\ref{sec:global-Lp}
and Theorem~\ref{thm:uniform-pointwise} in Section~\ref{sec:global-sup-outside}.

\begin{thm}[$L_1$-approximation \& maximum pointwise error]\label{thm:intro-Lp}
	Assume $(A)$. Let $t_0\in[0,T)$ and let $\emptyset\neq I\subseteq\R$ be an open interval such that
	\begin{enumerate}[\hspace{.5cm}(i)]\addtolength{\itemsep}{0.25\baselineskip}
		\item $a,b$ are twice continuously differentiable on $I$,
		\item $\forall\,x\in I\colon\, b(x)\neq 0$,
		\item $\PP(X(t_0)\in I)>0$.
	\end{enumerate}
	Then there exist constants $c,\gamma\in(0,\infty)$ such that for all $n\in\N$, for all $s_1,\dots s_n\in [0,T]$ 
	and for all measurable mappings $u\colon\R^{n+1}\to L_1([0,T])$ we have
	\begin{align*}
		\PP\bigl( \|X-u(X(0), W( s_1 ), \dots, W( s_n ) )\|_1 \geq c/\sqrt{n} \bigr) \geq \gamma.
	\end{align*}
	In particular, we have for all $n\in\N$ that
	\begin{align*}
		\inf_{\substack{ s_1,\dots,s_n\in [0,T]\\ u\colon \R^{n+1}\to L_1([0,T]) \text{ measurable} }}\hspace{-0.5cm}
			\EE\bigl[ \|X - u(X(0), W( s_1 ), \dots, W( s_n ) ) \|_1 \bigr]
			\geq c\gamma\cdot n^{-1/2}
	\end{align*}
	and
	\begin{align*}
		\inf_{\substack{ s_1,\dots,s_n\in [0,T]\\ u\colon \R^{n+1}\to C([0,T]) \text{ measurable} }}
			\biggl\{ \sup_{t\in[0,T]} \EE\Bigl[\bigl|X(t) - \bigl(u(X(0),W(s_1),\dots,W(s_n))\bigr)(t) \bigr|\Bigr] \biggr\}
			\geq c\gamma/T\cdot n^{-1/2}.
	\end{align*}
\end{thm}

Finally, we consider strong approximation of the solution at the final time.
The following result is an immediate consequence of Theorem~\ref{thm:ft-local} in Section~\ref{sec:ft-local}.

\begin{thm}[Pointwise approximation]\label{thm:intro-ft}
	Assume $(A)$. Let $t_0\in[0,T)$ and let $\emptyset\neq I\subseteq\R$ be an open interval such that
	\begin{enumerate}[\hspace{.5cm}(i)]\addtolength{\itemsep}{0.25\baselineskip}
		\item\label{thm3parti} $a,b$ are three times continuously differentiable on $I$,
		\item\label{thm3partii} $\forall\,x\in I\colon\, b(x) \neq 0 \text{ and } 
			\bigl(a'b - a\:\!b' - \frac{1}{2} b^2b''\bigr)(x) \neq 0$,
		\item\label{thm3partiii} $\PP(X(t_0)\in I)>0$.
	\end{enumerate}
	Then there exist constants $c,\gamma\in(0,\infty)$ such that for all $n\in\N$, for all $s_1,\dots s_n\in [0,T]$
	and for all measurable mappings $u\colon\R^{n+1}\to\R$ we have
	\begin{align*}
		\PP\bigl( |X(T)-u(X(0), W( s_1 ), \dots, W( s_n ) ) |\geq c/n \bigr) \geq \gamma.
	\end{align*}
	In particular, we have for all $n\in\N$ that
	\begin{align*}
		\inf_{\substack{ s_1,\dots,s_n\in [0,T]\\ u\colon\R^{n+1}\to\R \text{ measurable} }}
			\EE\bigl[ |X(T) - u(X(0), W( s_1 ), \dots, W( s_n ) ) | \bigr]
			\geq c\gamma\cdot n^{-1}.
	\end{align*}
\end{thm}

We stress that up to now the lower bounds on the mean errors in Theorems~\ref{thm:intro-sup}-\ref{thm:intro-ft}
were known to hold only under assumptions on the coefficients $a$ and $b$ and the initial value $X(0)$
that are much stronger than the conditions (i)-(iii) used in the above theorems.
For instance, all derivatives of $a$ and $b$,
which appear in Theorems~\ref{thm:intro-sup}-\ref{thm:intro-ft}, are typically assumed to exist on the whole real line
and to be bounded and the initial value $X(0)$ is required to satisfy a moment condition, see, e.g.,
\citet*{MG02,MG04,MG02_habil} for further details and references.
Moreover, Theorems~\ref{thm:intro-sup}-\ref{thm:intro-ft} provide error bounds that hold with positive probability, uniformly in $n$.
To the best of our knowledge such error estimates have not been established in this generality in the literature so far. 

The second condition in Assumption~\eqref{thm3partii} in Theorem~\ref{thm:intro-ft} requires some motivation,
which we take from a related discussion in \citet*{MG04}.
First note that the function $a'b-ab'-b^2 b''/2$
is the Lie bracket $[\widetilde a, b]$, where $\widetilde a = a-b b'/2$ is the drift coefficient of the Stratonovich equation
corresponding to the It\^{o} coefficients $a$ and $b$. By a general result of \citet*[Theorem~2.1]{Yamato1979},
which links the representability of the solution $X$ in terms of multiple It\^{o} integrals to the nilpotent property
of the Lie algebras associated with $\widetilde a$ and $b$, it follows that $X(T)= u(X(0),W(T))$ for some measurable function $u\colon\R^2\to\R$ if $\widetilde a, b\in C^\infty(\R)$ and $[\widetilde a, b]= 0$.
A special example is provided by a geometric Brownian motion, where $a(x) = \alpha x$, $b(x) = \beta x$ with $\alpha,\beta\in\R$, and $X(T) = X(0) \exp((\alpha-\beta^2/2)T +\beta W(T))$.
In particular, $X(T)$ can then be approximated with error zero based only on $X(0)$ and $W(T)$
and the lower error bound in Theorem~\ref{thm:intro-ft} cannot hold.
Thus, roughly speaking, the second condition in~\eqref{thm3partii} in Theorem~\ref{thm:intro-ft} (and its generalized version for non-autonomous equations in Proposition~\ref{prop:ft-global}) excludes
trivial approximation problems for SDEs such as a geometric Brownian motion or  
a one-dimensional squared Bessel process, see Section~\ref{sec:cir}.

Theorems~\ref{thm:intro-sup}-\ref{thm:intro-ft} yield lower error bounds for a huge variety of SDEs
including Cox-Ross-Ingersoll processes, see Section~\ref{sec:cir},
equations with superlinearly growing coefficients, see Section~\ref{sec:superlinear},
and equations with discontinuous coefficients, see Section~\ref{sec:discontinuous}.
In many of these cases these lower error bounds turn out to be sharp.  
Here, we illustrate our results by considering an SDE with a superlinearly growing drift coefficient, namely
\begin{align}\label{sde2}
	\begin{aligned}
		\dd X(t) &= -\bigl(X(t)\bigr)^5\,\dd t + X(t)\,\dd W(t), \qquad t\in [0,T],\\
		X(0) &= x_0\in \R\setminus\{0\}.
	\end{aligned}
\end{align}
Note that~\eqref{sde2} has a unique strong solution since both coefficients are locally Lipschitz continuous
and jointly satisfy a suitable monotone condition, see, e.g., \citet*{Mao2008}.
Clearly, $(a'b-ab'-b^2 b''/2)(x)=-4x^5$ for all $x\in\R$, so that all of the assumptions
in Theorems~\ref{thm:intro-sup}-\ref{thm:intro-ft} are satisfied
and therefore all of the respective lower error bounds hold true.
On the other hand, a tamed version of the equidistant Milstein scheme achieves the upper bound $1/n$, up to a constant,
for the error at the final time in the $p$-th mean, for any $p\in[1,\infty)$, see \citet*{Kumar-Sabanis2016},
while the piecewise linear interpolation of a tamed version of the Euler scheme achieves the upper bounds $n^{-1/2}$ and $\sqrt{\ln (n+1) /n}$, up to a constant each,
for the maximum pointwise error in the $p$-th mean (and thus the pathwise $L_p$-error in the $p$-th mean) and for the uniform error in the $p$-th mean, respectively,
for any $p\in[1,\infty)$, see \citet*{HJK12} and \citet*{Hutzenthaler-Jentzen-Kloeden2013}, respectively. Hence all of the lower error bounds from Theorems~\ref{thm:intro-sup}-\ref{thm:intro-ft}
are sharp and both methods just mentioned perform asymptotically optimal for equation \eqref{sde2}.

The present paper only addresses scalar SDEs and strong approximations based on evaluations of the driving Brownian motion
at single times in $[0,T]$. However, our proof techniques can also be applied in the case of systems of SDEs
and approximations based on bounded linear functionals of the driving Brownian motion or iterated It\^{o} integrals,
and we therefore expect an analogue transfer of the lower error bounds known for such methods in the case of systems of SDEs
with coefficients that are globally Lipschitz with respect to the state variable, see, e.g.,
\citet*{MG02_habil,Hofmann-MG-Ritter2002,Hofmann-MG2004}, to the case of systems of SDEs with coefficients
that behave sufficiently well with respect to the state variable only in a neighborhood of the initial value.

Our paper is organized as follows.
In Section~\ref{sec:notation} and Section~\ref{sec:setting} we introduce some notation and fix the setting with respect to equation~\eqref{sde1}, respectively.
In Section~\ref{sec:methods} we thoroughly explain what kind of approximation methods are covered by
the lower error bounds. In Section~\ref{sec:ft} we establish lower error bounds for pointwise approximation whereas  Section~\ref{sec:global} provides lower error bounds with respect to global approximation. In Section~\ref{sec:examples} we then use the results from the latter two sections to obtain lower error bounds for strong approximation of Cox-Ingersoll-Ross processes, equations with superlinearly growing coefficients,
and equations with discontinuous coefficients. The proofs of our lower error bounds rely on a localization technique, which provides a link between equations with coefficients that are globally Lipschitz in space and equations with coefficients that behave well only in a small neighbourhood of the initial value. This tool as well as a number of needed properties of Gaussian distributions are provided in an appendix.

\section{Notation}\label{sec:notation}

For $T\in(0,\infty)$ and $p\in [1,\infty)$ we use $L_p([0,T])$ to denote the Banach space
of all Borel-measurable functions $f\colon [0,T]\to\R$ with $\|f\|_p=\bigl(\int_0^T |f(t)|^p\,\dd t \bigr)^{1/p} < \infty$,
where two functions are identified if they coincide Lebesgue-almost everywhere on $[0,T]$.
By $C([0,T])$ we denote the Banach space of all continuous functions $f\colon [0,T]\to\R$
equipped with the norm $\|f\|_\infty = \sup_{t\in[0,T]}|f(t)|$. Furthermore, for $d\in\N$,
$x\in\R^d$, and $p\in [1,\infty)$ we use $|x|_p= (\sum_{i=1}^d|x_i|^p)^{1/p}$ to denote
the $p$-norm of $x$ and we use $|x|_\infty=\max_{i=1,\dots,d}|x_i|$ to denote the maximum norm of $x$.

\section{Setting}\label{sec:setting}

Throughout this article we fix the following scenario of a scalar SDE.
Let $(\Omega,\F,\PP)$ be a complete probability space with a normal filtration
$(\Fc_t)_{t\in[0,\infty)}$ and let $W\colon [0,\infty)\times\Omega\to\R$
be a standard Brownian motion on $(\Omega,\F,\PP)$ with respect to
$(\Fc_t)_{t\in [0,\infty)}$. Let $T\in (0,\infty)$, let $a\colon [0,T]\times\R\to\R$
and $b\colon [0,T]\times\R\to\R$ be Borel-measurable functions, and
let $X\colon [0,T]\times\Omega\to\R$ be an $(\Fc_t)_{t\in [0,T]}$-adapted
stochastic process with continuous sample paths such that $\PP$-a.s.~
$\int_0^T (|a(t,X(t))| + |b(t,X(t))|^2)\,\dd t < \infty$ and for all $t\in[0,T]$ $\PP$-a.s.~
\begin{align}\label{eq:SDE}
	X(t) = X(0) + \int_0^t a(s,X(s))\,\dd s + \int_0^t b(s,X(s))\,\dd W(s).
\end{align}

\section{The Class of Methods}\label{sec:methods}

Consider the setting in Section~\ref{sec:setting} and fix $n\in\N$.
We formally explain what we mean by an approximation of $X(T)$ or of $X$ that is based on the initial value $X(0)$
and on at most $n$ sequential evaluations of the Brownian motion $W$ at points in $[0,\infty)$ on average.

To this end we consider, more generally, a measurable space $(S,\Sc)$ 
and we introduce the class $\Ac_n(S,X(0),W)$
of all $\F$-$\Sc$-measurable mappings $V\colon \Omega \to S$ that can be constructed using $X(0)$
and at most $n$ sequential evaluations of the Brownian motion $W$ in $[0,\infty)$ on average.
In Section~\ref{sec:ft} we take $S=\R$ to study approximation of $X(T)$.
In Section~\ref{sec:global} we take $S=L_p([0,T])$ or  $S=C([0,T])$ to study
approximation of the whole process $X$. In either of these cases $\Sc$ is taken to be the
Borel $\sigma$-field generated by the respective canonical norm.

Every random variable $V\in \Ac_n(S,X(0),W)$ is determined by three sequences
\[
\psi = (\psi_k)_{k\in\N}, \quad \chi = (\chi_k)_{k\in\N}, \quad \varphi = (\varphi_k)_{k\in\N}
\] 
of measurable mappings
\begin{align*}
	\psi_k\colon & \R^k \to [0,\infty),\\
	\chi_k\colon & \R^{k+1} \to \{\text{STOP},\text{GO}\},\\
	\varphi_k\colon & \R^{k+1} \to S. 
\end{align*}
The sequence $\psi$ is used to determine the sequential evaluation sites for the Brownian motion $W$ in $[0,\infty)$.
The sequence $\chi$ determines when to stop the evaluation of $W$. The sequence $\varphi$ is used to obtain
the outcome of $V$ in $S$ once the evaluation of $W$ has stopped.

To be more precise, let $x\in\R$ be a possible realization of $X(0)$ and let $w\in C([0,\infty))$ be a possible realization of $W$.
We put $D_0(x,w)=x$, and for $k\in \N$ we recursively define 
\[
\tau_k(x,w) = \psi_k(D_{k-1}(x,w)), \quad Y_k(x,w)=w(\tau_k(x,w)),\quad D_k(x,w) = (D_{k-1}(x,w),Y_k(x,w)).
\] 
Thus $\tau_k(x,w)\in [0,\infty)$ is the $k$-th evaluation node for the actual path $w$ of $W$ and $Y_k(x,w)\in\R$
is the corresponding evaluation of $w$. The vector $D_k(x,w)\in \R^{k+1}$ contains the available data
about $x$ and $w$ after $k$ steps and we decide whether to stop or to go on with the sequential evaluation of $w$
according to the value of $\chi_k(D_{k}(x,w))$. The total number of evaluations of $w$ is thus given by
\[
\nu(x,w) = \min\{k\in\N\colon\, \chi_k(D_k(x,w)) = \text{STOP}\}\in \N\cup\{\infty\}.
\]
We require that the triple of sequences $(\psi,\chi,\varphi)$ satisfies $\PP(\nu(X(0),W)< \infty)=1$
and we define $V\colon\Omega\to S$ to be a random variable that satisfies $\PP$-a.s.
\begin{align}\label{eq:alg}
	V= \varphi_{\nu(X(0),W)}(D_{\nu(X(0),W)}(X(0),W)).
\end{align}
Now we put
\begin{align*}
	\Ac_n(S,X(0),W) = \{V\colon \Omega \to S\colon\, V \text{ is of the form \eqref{eq:alg} with }\EE[\nu(X(0),W)] \le n\}.
\end{align*}

\bigskip
The class $\Ac_n(S,X(0),W)$ is called the class of adaptive methods (with varying cardinality).
For $V\in\Ac_n(S,X(0),W)$, the corresponding average number $\EE[\nu(X(0),W)]$ of evaluation nodes of the
Brownian motion $W$ can be seen as a rough measure for the computational cost of the method $V$.
Prominent examples of such adaptive methods are Euler- or Milstein-type methods with a step size control
where the average number of evaluation nodes is bounded by $n$, see, e.g.,
\citet*{BHB04,Maut98,KSTZ05,GL97,HMGR00-1,HMGR00-2,HMGR01,MG02_habil,MG02,MG04}.

An important subclass  of $\Ac_n(S,X(0),W)$ is the class $\Ac^\mathrm{eq}_n(S,X(0),W)$ of all  methods that are based on $n$ evaluations of $W$ at the equidistant times $kT/n$, $k=0,1,\dots,n$. Formally, we have  $\chi_k=\text{GO}$ for $k< n$, $\chi_n=\text{STOP}$ and $\psi_k=Tk/\nu$ for $k\le n$. Thus
\[
\Ac^\mathrm{eq}_n(S,X(0),W) =\{u(X(0),W(T/n),W(2T/n),\dots,W(T))\colon u\colon \R^{n+1}\to S \text{ is measurable}\}
\]
and we have $\Ac^\mathrm{eq}_n(S,X(0),W)\subseteq \Ac_n(S,X(0),W)$. 
The lower error bounds established in Sections~\ref{sec:ft} and~\ref{sec:global} hold for any method from the class $\Ac_n(S,X(0),W)$ for the respective choice of $S$. In Section~\ref{sec:examples} we will see that matching upper error bounds are often (but not always) achieved by methods from the class $\Ac^\mathrm{eq}_n(S,X(0),W)$.

\bigskip
For technical reasons we introduce a further class of algorithms.
By $\Aa_n(S,W)$ we denote  the set of all random variables $V\colon\Omega\to S$
that are measurable with respect the $\sigma$-algebra generated by $\Fc_0,W(\tau_1),\dots, W(\tau_n)$,
where $\tau_1,\dots,\tau_n\colon \Omega\to [0,\infty)$ are any random variables
such that for all $k=1,\dots,n$ the random variable $\tau_k$ is measurable
with respect to the $\sigma$-algebra generated by $\Fc_0,W(\tau_1),\dots, W(\tau_{k-1})$.

\section{Lower Error Bounds for Strong Approximation at a Single Point}\label{sec:ft}

In this section we consider strong approximation of $X(T)$. In view of Section~\ref{sec:methods},
we study lower error bounds for approximation methods $\Xh_n(T)$ belonging to the class $\Ac_n(\R,X(0),W)$.
In Section~\ref{sec:ft-global}, we extend results from \citet*{MG04} and prove probability bounds on the error of $\Xh_n(T)\in \Aa_n(\R,W)$ under rather restrictive global smoothness assumptions on the coefficients $a$ and $b$ of the SDE~\eqref{eq:SDE}. In Section~\ref{sec:ft-local}, we switch to rather mild local smoothness assumptions on $a$ and $b$ and prove our main result on lower error bounds for pointwise approximation, see Theorem~\ref{thm:ft-local}.

\subsection{Global Assumptions}\label{sec:ft-global}

The following proposition shows that for sufficiently smooth coefficients $a$ and $b$, the probability for an adaptive method $\Xh_n(T)\in \Aa_n(\R,W)$ of having an error at least of order one is arbitrarily close to one.

\begin{prop}[Probability bounds on the pointwise error of adaptive methods I]\label{prop:ft-global}
	Assume the setting in Section~\ref{sec:setting} with $\EE{\left[|X(0)|^{16}\right]}<\infty$.
	Assume further that there exists a Borel set $I\subseteq\R$ such that
	\begin{enumerate}[\hspace{.5cm}({A}1)]\addtolength{\itemsep}{0.25\baselineskip}
		\item\label{cond:regularity}
			$\forall\, i\in\{0,1,2\},\, j\in\{0,1,2,3\}\colon\,$the partial derivatives $a^{(i,j)}, b^{(i,j)}$
			exist on $[0,T]\times \R$ and are continuous and bounded,
		\item\label{cond:ft} $\forall\,x\in I\colon\, {\left(a^{(0,1)}\,b - b^{(1,0)} -
			a\,b^{(0,1)} - \frac{1}{2} b^2 b^{(0,2)}\right)} (0,x) \neq 0$,
		\item\label{cond:init} $\PP(X(0)\in I) =1$.
	\end{enumerate}
	Then for every $ \varepsilon\in (0,1)$ there exists $ c\in (0,\infty)$ such that
	for all $n\in\N$ and for all $\Xh_n(T)\in\Aa_n(\R,W)$ we have
	\begin{align*}
		\PP \bigl(|X(T)-\Xh_n(T)|\geq c/n\bigr) \geq 1-\varepsilon.
	\end{align*}
\end{prop}

\begin{proof}
This proof is based on the techniques and the notation developed in \citet*{MG04} for the $p$-th mean error analysis
	of pointwise approximation of scalar SDEs.
	We provide the corresponding tools and results from the latter work that are needed in the present context.
	The assumption $\EE{\left[|X(0)|^{16}\right]}<\infty$ corresponds to the case $p=1$ in \citet*[Condition~(B), p.~1609]{MG04}.
	
	Without loss of generality we may assume $T=1$.
	For every $t\in [0,1]$ we put 
	\[
	\Ups(t)= \msf(t)\cdot \zsf(t,X(t))
	\]
	where
	\[
	\zsf =  a^{(0,1)}\,b - b^{(1,0)} - a\,b^{(0,1)} - \tfrac{1}{2} b^2 b^{(0,2)}
	\]
	and 
	\[
	\msf(t) = \exp\biggl(\int_t^1\Bigl(a^{(0,1)}-1/2\cdot\bigl(b^{(0,1)}\bigr)^2\Bigr)(u,X(u))\,\dd u 
	+\int_t^1 b^{(0,1)}(u,X(u))\,\dd W(u)\biggr).
	\]
	The analysis in \citet*{MG04} shows that the problem of pathwise approximation of $X(1)$
	based on finitely many evaluations of the driving Brownian motion $W$ is strongly connected
	to an integration problem for $W$ with a random weight given by the process $\Ups$.

	For $k\in\N$ we put $t_\ell = \ell/k$ for $\ell=0,\dots,k$ and we define a truncated Wagner-Platen scheme
	$\Xhwpt_k=\bigl(\Xhwpt_k(t_\ell)\bigr)_{\ell=0,\dots,k}$ by $\Xhwpt_k(0)=X(0)$ and for $\ell=0,\ldots,k-1$,
		\begin{align*}
		&\Xhwpt_k(t_{\ell+1})\\
		& \quad=\Xhwpt_k(t_\ell) + a\bigl(t_\ell,\Xhwpt_k(t_\ell)\bigr) \cdot (t_{\ell+1}-t_\ell) + 
			b\bigl(t_\ell,\Xhwpt_k(t_\ell)\bigr) \cdot\bigl(W(t_{\ell+1})-W(t_\ell)\bigr)\\
		& \qquad+ 1/2\cdot\bigl(bb^{(0,1)}\bigr)\bigl(t_\ell,\Xhwpt_k(t_\ell)\bigr)
			\cdot\bigl((W(t_{\ell+1})-W(t_\ell))^2-(t_{\ell+1}-t_\ell)\bigr)\\
		& \qquad+ 
			\bigl(b^{(1,0)}+ab^{(0,1)}-1/2\cdot b\,(b^{(0,1)})^2 \bigr)
			\bigl(t_\ell,\Xhwpt_k(t_\ell)\bigr) \cdot\bigl(W(t_{\ell+1})-W(t_\ell)\bigr)\cdot 
			(t_{\ell+1}-t_\ell)\\
		& \qquad+ 
			1/6\cdot\bigl(b(b^{(0,1)})^2+b^2b^{(0,2)}\bigr)
			\bigl(t_\ell,\Xhwpt_k(t_\ell)\bigr) \cdot\bigl(W(t_{\ell+1})-W(t_\ell)\bigr)^3\\
		& \qquad+ 1/2\cdot\bigl(a^{(1,0)}+aa^{(0,1)}+1/2\cdot b^2 
			a^{(0,2)}\bigr)\bigl(t_\ell,\Xhwpt_k(t_\ell)\bigr) \cdot (t_{\ell+1}-t_\ell)^2.
	\end{align*}
	The scheme $\Xhwpt_k$ is used for two purposes. First, we obtain a discrete-time
	approximation $\Upsh_k =\bigl(\Upsh_k(t_\ell)\bigr)_{\ell=0,\ldots,k-1}$
	to the random weight $\Ups$ in the following way. Put
	\[
	\Mh_\ell = 1+a^{(0,1)}\bigl(t_\ell,\Xhwpt_k(t_\ell)\bigr)\cdot (t_{\ell+1}-t_\ell) + 
		b^{(0,1)}\bigl(t_\ell,\Xhwpt_k(t_\ell)\bigr) \cdot \bigl(W(t_{\ell+1})-W(t_\ell)\bigr)
	\]
	for $\ell=0,\dots,k-1$, define a discrete-time Euler-type approximation
	$\msfh_k=\bigl(\msfh_k(t_\ell)\bigr)_{\ell=0,\dots,k}$ of the process $\msf$ by
	\[
	\msfh_k(t_\ell)=
		\begin{cases}
			\Mh_\ell\cdots \Mh_{k-1}, & \text{if }\ell\le  k-1,\\
			1, & \text{if } \ell=k,
		\end{cases}
	\]
	and define for $\ell=0,\ldots,k-1$, 
	\[
	\Upsh_k(t_\ell) = \msfh_k(t_{\ell+1})\cdot \zsf\bigl(t_\ell,\Xhwpt_k(t_\ell)\bigr).
	\]
	Second, we define an auxiliary scheme $\Xaux_k=\bigl(\Xaux_k(t_\ell)\bigr)_{\ell=0,\ldots,k}$ by
	\[
	\Xaux_k(t_\ell) = \Xhwpt_k(t_\ell) + \Qb_k(t_\ell),
	\]
	where $\Qb_k=\bigl(\Qb_k(t_\ell)\bigr)_{\ell=0,\dots,k}$ is defined by $\Qb_k(0)=0$ and 
	\begin{align*}
		\Qb_k(t_{\ell+1}) = \Mh_\ell\cdot \Qb_k(t_{\ell})
			+ \zsf\bigl(t_\ell,\Xhwpt_k(t_\ell)\bigr)
			\cdot \int_{t_\ell}^{t_{\ell+1}} \bigl(W(t)-W(t_\ell)\bigr)\,\dd t
	\end{align*}
	for $\ell=0,\dots,k-1$. 
	Observe that for $\ell=0,\dots,k$ it holds that
	\begin{align*}
		\Qb_k(t_{\ell}) = \sum_{r=0}^{\ell-1} \Bigg[
			\biggl(\zsf\bigl(t_r,\Xhwpt_k(t_r)\bigr)\cdot\int_{t_r}^{t_{r+1}}\bigl(W(t)-W(t_r)\bigr)\,\dd t\biggr)
			\cdot \prod_{j=r+1}^{\ell-1}\Mh_j \Bigg].
	\end{align*}
	In particular, we have
	\begin{align*}
		\Qb_k(1) &= \sum_{r=0}^{k-1} \Bigg[
			\biggl(\zsf\bigl(t_r,\Xhwpt_k(t_r)\bigr)\cdot\int_{t_r}^{t_{r+1}}\bigl(W(t)-W(t_r)\bigr)\,\dd t\biggr)
			\cdot \prod_{j=r+1}^{k-1}\Mh_j \Bigg] \\
		&= \sum_{r=0}^{k-1} \Bigg[
			\biggl(\zsf\bigl(t_r,\Xhwpt_k(t_r)\bigr)\cdot\int_{t_r}^{t_{r+1}}\bigl(W(t)-W(t_r)\bigr)\,\dd t\biggr)
			\cdot \msfh_k(t_{r+1}) \Bigg] \\
		&= \sum_{r=0}^{k-1} \Bigg[ \Upsh_k(t_r) \cdot \int_{t_r}^{t_{r+1}}\bigl(W(t)-W(t_r)\bigr)\,\dd t \Bigg].
	\end{align*}
	Hence we get
	\begin{align}\label{eq:new0}
		\Xaux_k(1) = \Xhwpt_k(1) + \sum_{r=0}^{k-1}
			\Bigg[ \Upsh_k(t_r) \cdot \int_{t_r}^{t_{r+1}}\bigl(W(t)-W(t_r)\bigr)\,\dd t \Bigg].
	\end{align}
	We have 
	\begin{align}\label{eq:old1}
		\exists\,c_1\in(0,\infty) \enspace \forall\,k\in\N \colon\,
			\max_{\ell=0,\dots,k-1}\EE\bigl[|\Ups(t_\ell)-\Upsh_k(t_\ell)|\bigr]  \leq c_1\cdot k^{-1/2},
	\end{align}
	see \citet*[Lemma~2, p.~1626]{MG04}, and
	\begin{align}\label{eq:old2}
		\exists\,c_2\in(0,\infty) \enspace \forall\,k\in\N\colon\,
			\EE\bigl[|X(1)-\Xaux_k(1)|\bigr] \leq c_2\cdot k^{-3/2},
	\end{align}
	see \citet*[Lemma~12, p.~1639]{MG04}.
	
	For $k\in\N$ we put 
	\[
	R_k = \frac{1}{k}\cdot \sum_{\ell=0}^{k-1} \big|\Ups(t_\ell)\big|^{2/3}, \qquad
		\Rh_k = \frac{1}{k}\cdot \sum_{\ell=0}^{k-1} \big|\Upsh_k(t_\ell)\big|^{2/3}.
	\]
	Note that the process $\Ups$ has continuous paths. Hence, $\PP$-a.s.,
	\begin{align}\label{eq:new1}
		\lim_{k\to\infty} R_k = \int_0^1 |\Ups(t)|^{2/3}\,\dd t.
	\end{align}
	Moreover, since 
	\begin{align*}
		\EE[|R_k-\Rh_k|] &= \frac{1}{k}\cdot \EE\!\left[\biggl\vert
				\sum_{\ell=0}^{k-1} \bigl|\Ups(t_\ell)\bigr|^{2/3}-\bigl|\Upsh_k(t_\ell)\bigr|^{2/3}
			\biggr\vert\right] \\
		&\leq \frac{1}{k}\cdot \sum_{\ell=0}^{k-1}\EE\!\left[\Bigl\vert
			\bigl|\Ups(t_\ell)\bigr|^{2/3}-\bigl|\Upsh_k(t_\ell)\bigr|^{2/3} \Bigr\vert\right] \\
		&\leq \frac{1}{k}\cdot \sum_{\ell=0}^{k-1}\EE\Bigr[\bigl|\Ups(t_\ell) - \Upsh_k(t_\ell)\bigr|^{2/3}\Bigr]
			\leq \max_{\ell=0,\dots,k-1} {\left(
				\EE\Bigr[\bigl|\Ups(t_\ell) - \Upsh_k(t_\ell)\bigr|\Bigr] \right)^{2/3}},
	\end{align*}
	we have 
	\begin{align}\label{eq:new2}
		\lim_{k\to\infty}\EE\bigl[|R_k-\Rh_k|\bigr] = 0
	\end{align}
	due to~\eqref{eq:old1}. Combining~\eqref{eq:new1} with~\eqref{eq:new2} we conclude that 
	\begin{align}\label{eq:new3}
		\lim_{k\to\infty} \Rh_k = \int_0^1 |\Ups(t)|^{2/3}\,\dd t\quad\text{in probability.}
	\end{align}
	By definition of $\Ups$, (A\ref{cond:ft}), and (A\ref{cond:init}) we have
	\begin{align*}
		\PP(\Ups(0)\neq0) = \PP(\zsf(0,X(0))\neq 0) \geq \PP(X(0)\in I) = 1.
	\end{align*}
	Observing the continuity of the process $\Ups$ we thus obtain 
	\begin{align}\label{eq:new4}
		\PP\biggl(\int_0^1 |\Ups(t)|^{2/3}\,\dd t > 0\biggr) = 1.
	\end{align}
	Combining~\eqref{eq:new3} with~\eqref{eq:new4} yields
	\begin{align}\label{eq:new5}
		\forall\,\varepsilon\in (0,1) \enspace \exists\, k_0\in\N, \beta \in (0,1) \enspace
			\forall\, k\ge k_0\colon\, \PP\bigl(\Rh_k \geq \beta\bigr)\ge 1-\varepsilon.
	\end{align}
	
	Let $k\in\N$ and consider an approximation $\Xh_k(1)\in \Aa_{k}(\R, W)$.
	Observe that there exist random variables $\tau_0,\dots,\tau_{2k}\colon \Omega\to [0,\infty)$ with
	\begin{enumerate}[(i)]\addtolength{\itemsep}{0.25\baselineskip}
		\item $\tau_\ell$ is measurable with respect to $\sigma\bigl(\Fc_0,W(\tau_0),\dots, W(\tau_{\ell-1})\bigr)$ for all $\ell=0,\dots,2k$,
		\item $\tau_\ell=t_\ell$ for all $\ell=0,\dots,k$,
		\item $\Xh_k(1)$ is measurable with respect to the $\sigma$-algebra $\A=\sigma\bigl(\Fc_0,W(\tau_0),\dots, W(\tau_{2k})\bigr)$.
	\end{enumerate}
	Let $B=(B(t))_{t\in[0,1]}$ denote the piecewise linear interpolation of $(W(t))_{t\in[0,1]}$ at the nodes $\tau_0,\dots,\tau_{2k}$.
	We define the process $Z=(Z(t))_{t\in[0,1]}$ by $Z(t)=W(t)-B(t)$ and put
	\[
	U = \Xh_k(1)-\Xhwpt_{k}(1)-\sum_{\ell=0}^{k-1} \Upsh_{k}(t_\ell)
		\cdot\int_{t_\ell}^{t_{\ell+1}} \bigl(B(t)-W(t_\ell)\bigr)\,\dd t.
	\]
	From \eqref{eq:new0} we get
	\begin{align*}
		\Xaux_{k}(1)-\Xh_k(1)
			= \sum_{\ell=0}^{k-1} \Upsh_{k}(t_\ell)\cdot\int_{t_\ell}^{t_{\ell+1}} Z(t)\,\dd t - U.
	\end{align*}
	Conditioned on $\A$, the values of $\Upsh_k$, $U$, and the evaluation sites
	$\tau_0,\dots,\tau_{2k}$ are fixed
	and the process $Z$ consists of independent Brownian bridges (from $0$ to $0$) on the subintervals corresponding to
	$\tau_0,\dots,\tau_{2k}$, cf.~\citet*[Lemma~1-2]{yaroslavtseva2017}. Note that
	\begin{align*}
		\EE\biggl[ \int_{t_\ell}^{t_{\ell+1}} Z(t)\,\dd t \,\Big|\, \A\biggr] = 0
	\end{align*}
	for $\ell=0,\ldots,k-1$. We conclude that conditioned on $\A$
	the random variable $\Xaux_{k}(1)-\Xh_k(1)$ is normally distributed with variance 
	\begin{align}\label{eq:new6}
		\Var\Bigl[\Xaux_{k}(1)-\Xh_k(1) \,\big|\, \A\Bigr] = \EE\biggl[\Bigl|\sum_{\ell=0}^{k-1} 
			\Upsh_{k}(t_\ell)\cdot\int_{t_\ell}^{t_{\ell+1}} Z(t)\,\dd t \Bigr|^2 \,\Big|\, \A\biggr].
	\end{align}
	For $\ell\in\{0,\dots,k-1\}$ we define $d_\ell=\#\{i\in\{0,\dots,2k\}\colon\,\tau_i\in(t_\ell,t_{\ell+1})\}$.
	By using \citet*[Eq.~(17), p.~1624]{MG04} we obtain
	\begin{align}\label{eq:new7}
		\begin{aligned}[c]
			\EE\biggl[\Bigl|\sum_{\ell=0}^{k-1} 
				\Upsh_{k}(t_\ell)\cdot\int_{t_\ell}^{t_{\ell+1}} Z(t)\,\dd t \Bigr|^2 \,\Big|\, \A\biggr]
				&= \sum_{\ell=0}^{k-1}\left(\Upsh_{k}(t_\ell)\right)^2 \cdot
				\EE\biggl[\Bigl|\int_{t_\ell}^{t_{\ell+1}} Z(t)\,\dd t\Bigr|^2 \,\Big|\, \A\biggr] \\
			&\geq \sum_{\ell=0}^{k-1}\left(\Upsh_{k}(t_\ell)\right)^2 \cdot
				\frac{1}{12k^3} \cdot \frac{1}{\left(d_\ell+1\right)^2}.
		\end{aligned}
	\end{align}
	Clearly, $\sum_{\ell=0}^{k-1} d_\ell \leq  k$, and therefore, by the H\"older inequality,
	\[
		\bigl(k\cdot\Rh_k\bigr)^3
			= \left(\sum_{\ell=0}^{k-1} (d_\ell +1)^{2/3}
			\cdot \frac{\bigl|\Upsh_k(t_\ell)\bigr|^{2/3}}{(d_\ell +1)^{2/3}} \right)^3 \leq 4k^2 \cdot \sum_{\ell=0}^{k-1}\bigl(\Upsh_{k}(t_\ell)/(d_\ell +1)\bigr)^2,
	\]
	which jointly with~\eqref{eq:new6} and~\eqref{eq:new7} implies
	\begin{align*}
		\Var\bigl[k \cdot |\Xaux_{k}(1)-\Xh_k(1)| \,\big|\, \A\bigr]
		\geq \frac{1}{48}\,\Rh_k^3.
	\end{align*}
	Employing Lemma~\ref{lem:normal1} we conclude that
	\begin{align*}
		\forall\,\varepsilon\in (0,1)\colon\,
			\PP\bigl(k\cdot|\Xaux_{k}(1)-\Xh_k(1)|\geq\varepsilon\,(\Rh_k^{3}/48)^{1/2} \,\big|\, \A\bigr) \geq 1-\varepsilon
	\end{align*}
	and hence
	\begin{align}\label{eq:new8}
		\forall\,\varepsilon\in (0,1)\colon\,
			\PP\bigl(k\cdot|\Xaux_{k}(1)-\Xh_k(1)|\geq\varepsilon\,\Rh_k^{3/2}/7 \bigr) \geq 1-\varepsilon. 
	\end{align}

	Fix $\varepsilon\in(0,1)$ and choose $k_0\in\N$ and $\beta\in(0,1)$ according to~\eqref{eq:new5}.
	By~\eqref{eq:new5} and \eqref{eq:new8} we obtain for all $k\geq k_0$ and $\Xh_k(1)\in\Aa_{k}(\R, W)$ that
	\begin{align}\label{eq:new9}
		\begin{aligned}
			&\PP\Bigl(|\Xaux_{k}(1)-\Xh_k(1)| \geq \frac{\varepsilon\beta^{3/2}}{7k} \Bigr) \\
			&\qquad\geq \PP\bigl( \bigl\{ k\cdot|\Xaux_{k}(1)-\Xh_k(1)|\ge \varepsilon\,\Rh_k^{3/2}/7 \bigr\}
				\cap \{\Rh_k\geq\beta \} \bigr)\\
			&\qquad\geq \PP\bigl( k\cdot|\Xaux_{k}(1)-\Xh_k(1)|\ge \varepsilon\,\Rh_k^{3/2}/7 \bigr)
				- \PP\bigl( \Rh_k<\beta \bigr)\\
			&\qquad\geq 1-2\varepsilon.
		\end{aligned}
	\end{align}
	Put $c_3 = \varepsilon\beta^{3/2} /14\in (0,\infty)$. Using the Markov inequality we derive from~\eqref{eq:old2}
	that for all $k\in\N$ we have
	\begin{align}\label{eq:new10}
		\PP(|X(1)-\Xaux_{k}(1)|\ge c_3/k) \leq k/c_3\cdot c_2\cdot k^{-3/2}
		=c_2/c_3\cdot k^{-1/2}.
	\end{align}
	We conclude from \eqref{eq:new9} and \eqref{eq:new10} that for all
	$k\geq \max\bigl(k_0, \left(c_2/(c_3\varepsilon)\right)^2\bigr)$
	and $\Xh_k(1)\in\Aa_{k}(\R, W)$ we have
	\begin{align*}
		&\PP(|X(1)-\Xh_k(1)|\ge c_3/k)\\
		&\qquad\geq \PP\bigl( \bigl\{|\Xaux_{k}(1)-\Xh_k(1)|\geq 2c_3/k\bigr\}
			\cap \bigl\{|X(1)-\Xaux_{k}(1)|< c_3/k\bigr\} \bigr)\\
		&\qquad\geq \PP\bigl( |\Xaux_{k}(1)-\Xh_k(1)|\geq 2c_3/k \bigr) - \PP\bigl( |X(1)-\Xaux_{k}(1)|\geq c_3/k \bigr)\\
		&\qquad\geq 1-2\varepsilon - c_2/c_3\cdot k^{-1/2} \ge 1-3\varepsilon,
	\end{align*}
	which completes the proof.
\end{proof}

\subsection{Local Assumptions}\label{sec:ft-local}

The next proposition shows that even under very mild local regularity conditions on the 
coefficients $a$ and $b$, the probability for an adaptive method $\Xh_n(T)\in \Aa_n(\R,W)$ of having an error at least of order one is still uniformly bounded away from zero.
Its proof exploits a comparison result for SDEs, see Proposition~\ref{prop:localization} in the appendix,
to reduce the general case to the case treated in Proposition~\ref{prop:ft-global}.

\begin{prop}[Probability bounds on the pointwise error of adaptive methods II]\label{prop:ft-local}
	Assume the setting in Section~\ref{sec:setting}. Let $t_0\in[0,T)$
	and let $\emptyset\neq I\subseteq\R$ be an open interval such that
	\begin{enumerate}[\hspace{.5cm}({A}1*)]\addtolength{\itemsep}{0.25\baselineskip}
		\item\label{cond:regularity2} $\forall\, i\in\{0,1,2\},\, j\in\{0,1,2,3\}\colon\,$the partial derivatives
			$a^{(i,j)}, b^{(i,j)}$ exist on $[t_0,T]\times I$ and are continuous,
		\item\label{cond:ft2} $\forall\,x\in  I\colon\,
			{\left(a^{(0,1)}\,b - b^{(1,0)} - a\, b^{(0,1)} - \frac{1}{2} b^2 b^{(0,2)}\right)}(t_0,x)\neq 0$, and \\[.1cm] $\forall\, (t,x)\in [t_0,T]\times I\colon\, b(t,x) \neq 0$,
		\item\label{cond:init2} $\PP(X(t_0)\in I)>0$.
	\end{enumerate}
	Then there exist constants $c,\gamma\in(0,\infty)$ such that
	for all $n\in\N$ and for all $\Xh_n(T)\in\Aa_n(\R,W)$ we have
	\begin{align*}
		\PP \bigl(|X(T)-\Xh_n(T)|\geq c/n\bigr) \geq\gamma.
	\end{align*}
\end{prop}

\begin{proof}
	By considering the SDE~\eqref{eq:SDE} starting from time $t_0$ we may assume $t_0=0$.
	According to the openness of $I$ and (A\ref{cond:init2}*), there exist
	bounded open intervals $I_1,I_2,I_3\subseteq\R$ such that
	\begin{align*}
		\emptyset\neq I_3 \subseteq \bar{I_3} \subseteq I_2 \subseteq \bar{I_2} \subseteq I_1 \subseteq \bar{I_1} \subseteq I
	\end{align*}
	and $\PP(X(0)\in I_3) > 0$.
	Here, $\bar{I_1}, \bar{I_2},\bar{I_3}$ denote the closures of the intervals $I_1,I_2,I_3$, respectively.
	Due to the continuity of $b$ on $[0,T]\times \bar I_1$ and the second condition in (A\ref{cond:ft2}*) we may without loss of generality assume that
	\begin{align}\label{eq:b-pos}
		\inf_{(t,x)\in [0,T]\times \bar I_1} b(t,x) > 0.
	\end{align}
	Let $\eta_1,\eta_2\colon\R\to\R$ be infinitely differentiable functions
	such that $0\leq \eta_1,\eta_2 \leq 1$ and
	\begin{align*}
		\eta_1(x) =
		\begin{cases}
			1,  &\text{if }x\in I_2,\\
			0, &\text{if }x\in I_1^c,
		\end{cases}\qquad
		\eta_2(x) =
		\begin{cases}
			0, &\text{if }x\in I_3,\\
			1, &\text{if }x\in I_2^c.
		\end{cases}
	\end{align*}
	Furthermore, define $\at\colon[0,T]\times\R\to\R$ and $\bt\colon[0,T]\times\R\to\R$ by
	\begin{align*}
		\at(t,x)=\eta_1(x)\cdot a(t,x), \qquad \bt(t,x)=\eta_1(x)\cdot b(t,x)+\eta_2(x).
	\end{align*}
	Due to (A\ref{cond:regularity2}*) and \eqref{eq:b-pos} it holds that
	\begin{enumerate}[(i)]\addtolength{\itemsep}{0.25\baselineskip}
		\item $\forall\, i\in\{0,1,2\},\, j\in\{0,1,2,3\}\colon\,$the partial derivatives
			$\at^{(i,j)}, \bt^{(i,j)}$ exist on $[0,T]\times \R$ are continuous and bounded,
		\item $\inf_{(t,x)\in [0,T]\times \R} |\bt(t,x)| > 0$,
		\item $\forall\,t\in [0,T],\, x\in I_3\colon\, a(t,x)=\at(t,x)$ and $b(t,x)=\bt(t,x)$.
	\end{enumerate}
	Let $x_0\in I_3$ and define the bounded random variable $\Xb(0)\colon\Omega\to\R$ by
	\begin{align*}
		\Xb(0)(\omega) =
		\begin{cases}
			X(0)(\omega), &\text{if }X(0)(\omega)\in I_3,\\
			x_0, &\text{otherwise}.
		\end{cases}
	\end{align*}
	Furthermore, let $\Xb\colon [0,T]\times\Omega\to\R$ be an $(\Fc_t)_{t\in [0,T]}$-adapted stochastic process
	with continuous sample paths such that for all $t\in[0,T]$ it holds $\PP$-a.s.~that 
	\begin{align*}
		\Xb(t) = \Xb(0) + \int_0^t \at\bigl(s,\Xb(s)\bigr)\,\dd s + \int_0^t \bt\bigl(s,\Xb(s)\bigr)\,\dd W(s).
	\end{align*}
	Applying Proposition~\ref{prop:localization} with  $I=I_3$ shows
	\begin{align*}
		\PP\big( \forall\,t\in[0,T]\colon \Xb(t)=X(t) \big)>0
	\end{align*}
	and hence
	\begin{align}\label{eq:ft-prob1}
		\PP\big( \Xb(T)=X(T) \big)>0.
	\end{align}
	Moreover, Proposition~\ref{prop:ft-global} with $X=\Xb$, $a=\at$, $b=\bt$, $I=I_3$, $\varepsilon=\PP(\Xb(T)=X(T))/2$
	yields the existence of a constant $c\in (0,\infty)$ such that for all $n\in\N$ and for all $\Xh_n(T)\in\Aa_n(\R,W)$ we have
	\begin{align}\label{eq:ft-prob2}
		\PP(|\Xb(T)-\Xh_n(T)|\geq c/n) \geq 1- \PP\big(\Xb(T)=X(T)\big)/2.
	\end{align}
	Combining \eqref{eq:ft-prob1} with \eqref{eq:ft-prob2} shows that for all $n\in\N$ and for all $\Xh_n(T)\in \Aa_n(\R,W)$ it holds that
	\begin{align*}
		\PP\bigl( |X(T)-\Xh_n(T)| \geq c/n \bigr)
			&\geq \PP\bigl( \{|X(T)-\Xh_n(T)|\geq c/n\} \cap \{X(T) = \Xb(T)\} \bigr) \\
		&= \PP\bigl( \{|\Xb(T)-\Xh_n(T)|\geq c/n\} \cap \{X(T) = \Xb(T)\} \bigr) \\
		&\geq \PP\bigl( |\Xb(T)-\Xh_n(T)|\geq c/n \bigr) + \PP\bigl( X(T)=\Xb(T) \bigr) -1 \\
				&\ge \PP\bigl( \Xb(T)=X(T) \bigr)/2 > 0,
	\end{align*}
	which completes the proof.
\end{proof}

Proposition~\ref{prop:ft-local} provides a lower error bound for methods from the class $\Aa_n(\R,W)$, i.e., for adaptive methods that are based on $n$ evaluations of $W$.
We now extend this result to the class $\Ac_n(\R,X(0),W)$ of adaptive methods that may us
$n$ evaluations of $W$, on average, which is our main result for pointwise approximation.

\begin{thm}[Lower error bound for pointwise approximation]\label{thm:ft-local}
Assume the setting in Section~\ref{sec:setting}. Let $t_0\in[0,T)$
	and let $\emptyset\neq I\subseteq\R$ be an open interval such that
	the conditions~(A\ref{cond:regularity2}*),~(A\ref{cond:ft2}*), and~(A\ref{cond:init2}*) from Proposition~\ref{prop:ft-local} are satisfied.
	Then there exist constants $\bar{c},\bar{\gamma}\in(0,\infty)$ such that for all $n\in\N$
	and for all $\Xh_n(T)\in\Ac_n(\R,X(0),W)$ we have
	\begin{align*}
		\PP\bigl( |X(T)-\Xh_n(T)|\geq \bar{c}/n \bigr) \geq \bar{\gamma}.
	\end{align*}
	In particular, for $\hat{c}=\bar{c}\cdot\bar{\gamma}\in(0,\infty)$ we have for all $n\in\N$ that
	\begin{align*}
		\inf_{\Xh_n(T)\in\Ac_n(\R,X(0),W)} {\left\{ \EE\Bigl[ \bigl|X(T)-\Xh_n(T)\bigr| \Bigr] \right\}}
			\geq \hat{c}\cdot n^{-1}.
	\end{align*}
\end{thm}

\begin{proof}
Choose $c,\gamma\in(0,\infty)$ according to Proposition~\ref{prop:ft-local}
	and choose $k\in\N$ such that $1/k\leq \gamma/2$.
	Let $n\in\N$ and $\Xh_n(T)\in\Ac_n(\R,X(0),W)$,
	and let $\nu(X(0),W)$ denote the number of evaluations nodes used by $\Xh_n(T)$, see Section~\ref{sec:methods}.
	By assumption we have $\EE[\nu(X(0),W)]\leq n$.
	Hence the Markov inequality shows
	\begin{align}\label{eq:markov}
		\PP\bigl( \nu(X(0),W)\geq kn \bigr) \leq \frac{\EE[\nu(X(0),W)]}{kn} \leq 1/k \leq \gamma/2.
	\end{align}
	Define  $\Xh_{kn}^*(T)\in\Aa_{kn}(\R,W)$ by
	\begin{align*}
		\Xh_{kn}^*(T)(\omega) =
		\begin{cases}
			\Xh_n(T)(\omega), &\text{if }\nu(X(0),W)(\omega)<kn,\\
			0, &\text{otherwise},
		\end{cases}
	\end{align*}
	Then
	\begin{align*}
		\PP\bigl( \Xh_n(T)=\Xh_{kn}^*(T) \bigr) \geq \PP\bigl( \nu(X(0),W)<kn \bigr) \geq 1-\gamma/2
	\end{align*}
	due to \eqref{eq:markov}. Combining the latter fact with Proposition~\ref{prop:ft-local} yields
	\begin{align*}
		&\PP\bigl( |X(T)-\Xh_n(T)| \geq (c/k)/n \bigr) \\
		&\qquad\geq \PP\bigl( \{|X(T)-\Xh_n(T)|\geq c/(kn)\} \cap \{\Xh_n(T)=\Xh_{kn}^*(T)\} \bigr) \\
		&\qquad= \PP\bigl( \{|X(T)-\Xh_{kn}^*(T)|\geq c/(kn)\} \cap \{\Xh_n(T)=\Xh_{kn}^*(T)\} \bigr) \\
		&\qquad\geq \PP\bigl( |X(T)-\Xh_{kn}^*(T)|\geq c/(kn) \bigr) + \PP\bigl( \Xh_n(T)=\Xh_{kn}^*(T) \bigr) -1 \\
		&\qquad\geq \gamma/2 > 0,
	\end{align*}
	which completes the proof.
\end{proof}

\section{Lower Error Bounds for Strong Approximation Globally in Time}\label{sec:global}

In this section we consider strong approximation of $X$. In view of Section~\ref{sec:methods},
we study lower error bounds of approximation methods $\Xh_n$ belonging to the class $\Ac_n(C([0,T]),X(0),W)$ or $\Ac_n(L_p([0,T]),X(0),W)$.
In both cases we proceed similar to our analysis of one-point approximation in Section~\ref{sec:ft}. We first
prove probability bounds on the error of $\Xh_n$ under restrictive global smoothness assumptions on the coefficients $a$ and $b$ of the SDE~\eqref{eq:SDE} and then switch to the setting of mild local smoothness conditions by employing the localization technique from Appendix~\ref{sec:comparison-result}.

\subsection{$L_\infty$-Approximation}\label{sec:global-sup}

In this section we consider approximation with respect to the maximum distance
on the time interval $[0,T]$.
The following proposition extends mean error bounds from \citet*{MG02} to probability bounds for the error of $\Xh_n$ from $\Aa_n(C([0,T]),W)$.

\begin{prop}[Probability bounds on the $L_\infty$-error of adaptive methods I]\label{prop:sup-global}
	Assume the setting in Section~\ref{sec:setting} with $\EE{\left[|X(0)|^2\right]}<\infty$.
	Assume further that
	\begin{enumerate}[\hspace{.5cm}({B}1)]\addtolength{\itemsep}{0.25\baselineskip}
		\item\label{cond:regularity-sup}
			$\forall\,i\in\{0,1\},\,j\in\{0,1\}\colon\,$the partial derivatives $a^{(i,j)}, b^{(i,j)}$
			exist on $[0,T]\times\R$ and are continuous and bounded,
		\item\label{cond:nondeg-sup} $\inf_{(t,x)\in [0,T]\times\R} |b(t,x)|>0$.
	\end{enumerate}
	Then for every $\varepsilon\in (0,1)$ there exists $c\in(0,\infty)$ such that
	for all $n\in\N$ and for all $\Xh_n\in\Aa_n(C([0,T]),W)$ we have
	\begin{align*}
		\PP\bigl( \|X-\Xh_n\|_\infty\geq c\,\sqrt{\ln (n+1)/n} \bigr) \geq 1-\varepsilon.
	\end{align*}
\end{prop}

\begin{proof}
	Due to a scaling in time we may assume $T=1$.
	For $k\in\N$ we put $t_{\ell}=\ell/k$ for $\ell=0,\dots,k$ and we define a
	continuous-time Euler scheme $\Xhe_k=\bigl(\Xhe_k(t)\bigr)_{t\in [0,1]}$ by $\Xhe_k(0)=X(0)$
	and for $\ell=0,\dots,k-1$ and $t\in (t_\ell,t_{\ell+1}]$ by
	\[
	\Xhe_k(t) = \Xhe_k(t_{\ell}) + a\bigl(t_\ell,\Xhe_k(t_{\ell})\bigr)\cdot (t-t_{\ell})
		+ b\bigl(t_\ell,\Xhe_k(t_{\ell})\bigr)\cdot \bigl(W(t)-W(t_{\ell})\bigr).
	\]
	Moreover, for $\ell =0,\dots, k$ we put
	\[
	\bh_{k,\ell} = b\bigl( t_{\ell},\Xhe_k(t_{\ell}) \bigr).
	\]
	By (B\ref{cond:regularity-sup}) and $\EE{\left[|X(0)|^2\right]}<\infty$ we have
	\begin{align}\label{ext1}
		\exists\,c_1\in(0,\infty) \enspace \forall\,k\in\N\colon\, \EE\bigl[\|X-\Xhe_k\|_\infty\bigr]\le c_1 /\sqrt{k},
	\end{align} 
	see, e.g., \citet*[Theorem~3, page~631]{HMGR00-2}.

	Let $\varepsilon\in(0,1)$ and
	\begin{align*}
		\delta = \inf_{(t,x)\in [0,1]\times\R} |b(t,x)|\in (0,\infty),
	\end{align*}
	see (B\ref{cond:nondeg-sup}), and choose $c_2\in (0,\infty)$ according to Lemma~\ref{lem:normal2}.
	
	Let $k\in 2\N$ and consider an approximation $\Xh_k \in\Aa_{k/2}(C([0,1]), W)$.
	Observe that there exist random variables $\tau_0,\dots,\tau_{3k/2}\colon \Omega\to [0,\infty)$ with
	\begin{enumerate}[(i)]\addtolength{\itemsep}{0.25\baselineskip}
		\item $\tau_\ell$ is measurable with respect to $\sigma\bigl(\Fc_0,W(\tau_0),\dots, W(\tau_{\ell-1})\bigr)$ for all $\ell=0,\dots,3k/2$,
		\item $\tau_\ell=t_\ell$ for all $\ell=0,\dots,k$,
		\item $\Xh_k$ is measurable with respect to the $\sigma$-algebra $\A=\sigma\bigl(\Fc_0,W(\tau_0),\dots, W(\tau_{3k/2})\bigr)$.
	\end{enumerate}
	Let $B=(B(t))_{t\in[0,1]}$ denote the piecewise linear interpolation of $(W(t))_{t\in [0,1]}$ at the nodes $\tau_0,\dots,\tau_{3k/2}$.
	We define the process $Z=(Z(t))_{t\in [0,1]}$ by $Z(t)=W(t)-B(t)$. Moreover, we define the process $U=(U(t))_{t\in [0,1]}$ by
	\begin{multline*}
		U(t) = \Xh_k(t) -\biggl[ \ind_{\{0\}}(t)\cdot X(0)\\
		+ \sum_{\ell=0}^{k-1}\ind_{(t_{\ell},t_{\ell+1}]}(t)\cdot\Bigl(\Xhe_k(t_{\ell})
			+ a\bigl(t_{\ell},\Xhe_k(t_{\ell})\bigr)\cdot(t-t_{\ell}) + \bh_{k,\ell}\cdot \bigl(B(t)-W(t_{\ell})\bigr) \Bigr) \biggr].
	\end{multline*}
	By definition of $\Xhe_{k}$ it holds for all $t\in [0,1]$ that
	\[
	\Xhe_{k}(t)-\Xh_k(t) = \sum_{\ell=0}^{k-1}\Bigl[ \ind_{(t_{\ell},t_{\ell+1}]}(t)\cdot\bh_{k,\ell}\cdot Z(t) \Bigr] - U(t).
	\]
	Let 
	\[
	\mathcal L_k = \bigr\{\ell\in\{0,\dots,k-1\}\colon\, (t_{\ell},t_{\ell+1}) \cap \{\tau_0,\dots, \tau_{3k/2}\} =\emptyset\bigr\}
	\]
	and observe that
	\begin{align}\label{atleasthalf}
		\#\mathcal L_k\geq k/2.
	\end{align}
	Conditioned on $\A$, the values of $\bigl(\Xhe_k(t_\ell)\bigr)_{\ell\in\{0,\dots,k\}}$,
	$U$, $\bigl(\bh_{k,\ell}\bigr)_{\ell\in\{0,\dots,k\}}$, and the evaluation sites $\tau_0 ,\dots,\tau_{3k/2}$ used by $\Xh_k$
	are fixed and the process $Z$ consists of independent Brownian bridges (from $0$ to $0$)
	on the subintervals corresponding to $\tau_0,\dots\tau_{3k/2}$, cf.~\citet*[Lemma~1-2]{yaroslavtseva2017}.
	We conclude that conditioned on $\A$ the set $\mathcal L_k$ is fixed and the random variables
	\begin{align*}
		V_\ell = \bigl(\Xhe_{k}-\Xh_k\bigr)((t_{\ell}+t_{\ell+1})/2),
	\end{align*}
	where $\ell\in\{0,\dots,k-1\}$, are independent and normally distributed with variances satisfying
	\begin{align}\label{ext4}
		\Var \bigl(V_\ell \,\big|\, \A \bigr) \geq  \ind_{\mathcal L_k}(\ell)\cdot \frac{\delta^2}{4k}. 
	\end{align}
	Since $\|\Xhe_{k}-\Xh_k\|_\infty \geq \max_{\ell\in \mathcal L_k}|V_\ell|$, we get
	\begin{align}\label{eq10000}
		\PP\bigl( 2\sqrt{k}\cdot \|\Xhe_{k}-\Xh_k\|_\infty \geq c_2\sqrt{\ln(k/2)} \,\big|\, \A \bigr)
		\geq \PP\bigl( 2\sqrt{k}\cdot \max_{\ell\in\mathcal L_k}|V_\ell| \geq c_2\sqrt{\ln(k/2)} \,\big|\, \A \bigr).
	\end{align}
	Lemma~\ref{lem:normal2} and \eqref{atleasthalf} imply 
	\begin{align}\label{eq002}
		\PP\bigl( 2\sqrt{k}\cdot \max_{\ell\in\mathcal L_k}|V_\ell| \geq c_2\sqrt{\ln(k/2)} \,\big|\, \A \bigr)
			\geq 1-\varepsilon.
	\end{align}
	Combing \eqref{eq10000} and \eqref{eq002} yields
	\begin{align*}
		\PP\bigl(2\sqrt{k}\cdot \|\Xhe_{k} - \Xh_k\|_\infty  \geq c_2\sqrt{\ln(k/2)} \,\big|\, \A \bigr) \geq 1-\varepsilon
	\end{align*}
	and hence
	\begin{align}\label{lowerbound}
		\PP\bigl(2\sqrt{k}\cdot \|\Xhe_{k}-\Xh_k\|_\infty  \geq c_2\sqrt{\ln(k/2)}\bigr) \geq 1-\varepsilon.
	\end{align}
	By \eqref{lowerbound}, there exists a constant $c_3\in (0,\infty)$
	such that for all $k\in\N$ and $\Xh_k\in\Aa_{k}(C([0,1]), W)$ it holds that
	\begin{align*}
		\PP\bigl( \|\Xhe_{k}-\Xh_k\|_\infty \geq c_3\sqrt{\ln(k+1)/k} \bigr) \geq 1-\varepsilon.
	\end{align*}
	The latter fact and the Markov inequality combined with \eqref{ext1} imply that for all $k\in\N$
	and $\Xh_k \in\Aa_{k}(C([0,1]),W)$ it holds that
		\begin{align*}
			\PP\bigl( \|X&-\Xh_k\|_\infty \geq c_3/2\cdot \sqrt{\ln(k+1)/k} \bigr) \\
			&\geq \PP\bigl( \bigl\{\|\Xhe_{k}-\Xh_k\|_\infty \geq c_3\sqrt{\ln(k+1)/k} \bigr\}
				\cap \bigl\{\|X-\Xhe_{k}\|_\infty< c_3/2\cdot \sqrt{\ln(k+1)/k}\bigr\} \bigr) \\
			&\geq \PP\bigl( \|\Xhe_{k}-\Xh_k\|_\infty \geq c_3\sqrt{\ln(k+1)/k} \bigr)
				- \PP\bigl( \|X-\Xhe_{k}\|_\infty\geq c_3/2\cdot \sqrt{\ln(k+1)/k} \bigr) \\
			&\geq 1-\varepsilon - 2/c_3\cdot\sqrt{k/\ln(k+1)}\cdot c_1/\sqrt{k}\\
			&=1-\varepsilon-2c_1/\bigl( c_3\sqrt{\ln(k+1)} \bigr),
	\end{align*}
	which completes the proof.
\end{proof}

Combining the localization technique from Appendix~\ref{sec:comparison-result}
with Proposition~\ref{prop:sup-global} leads to the following result.

\begin{prop}[Probability bounds on the $L_\infty$-error of adaptive methods II]\label{prop:sup-local}
	Assume the setting in Section~\ref{sec:setting}.
	Let $t_0\in[0,T)$, $T_0\in (t_0,T]$ and let $\emptyset\neq I\subseteq\R$ be
	an open interval such that
	\begin{enumerate}[\hspace{.5cm}({B}1*)]\addtolength{\itemsep}{0.25\baselineskip}
		\item\label{cond1} $\forall\,i\in\{0,1\},\,j\in\{0,1\}\colon\,$the partial derivatives $a^{(i,j)}, b^{(i,j)}$
			exist on $[t_0,T_0]\times I$ and are continuous,
		\item\label{cond2} $\forall\,(t,x)\in [t_0,T_0]\times I\colon\, b(t,x)\neq 0$,
		\item\label{cond3} $\PP(X(t_0)\in I)>0$.
	\end{enumerate}
	Then there exist constants $c,\gamma\in(0,\infty)$ such that
	for all $n\in\N$ and for all $\Xh_n\in\Aa_n(C([0,T]),W)$ we have
	\begin{align*}
		\PP\bigl( \|X-\Xh_n\|_\infty \geq c\,\sqrt{\ln(n+1)/n} \bigr) \geq \gamma.
	\end{align*}
\end{prop}

\begin{proof}
	Since the error is measured globally in time, we may assume $T_0=T$. Now, the reasoning is almost identical to the reasoning in the proof of Proposition~\ref{prop:ft-local}.
	Instead of Proposition~\ref{prop:ft-global} we rely on Proposition~\ref{prop:sup-global}.
\end{proof}

Proposition~\ref{prop:sup-local} provides a lower error bound for methods from the class $\Aa_n(C([0,T]),W)$, i.e., for adaptive methods that are based on $n$ evaluations of $W$. 
We now extend this result to the class $\Ac_n(C([0,T]),X(0),W)$ of adaptive methods that may us $n$ evaluations of $W$, on average.
The following theorem is our main result for $L_\infty$-Approximation.

\begin{thm}[Lower error bound for $L_\infty$-approximation]\label{thm:sup-local}
Assume the setting in Section~\ref{sec:setting}.
	Let $t_0\in[0,T)$, $T_0\in (t_0,T]$ and let $\emptyset\neq I\subseteq\R$ be
	an open interval such that
	the conditions~(B\ref{cond1}*),~(B\ref{cond2}*), and~(B\ref{cond3}*) from Proposition~\ref{prop:sup-local} are satisfied.
	Then there exist constants $\bar{c},\bar{\gamma}\in(0,\infty)$ such that for all $n\in\N$
	and for all $\Xh_n\in\Ac_n(C([0,T]),X(0),W)$ we have
	\begin{align*}
		\PP\bigl( \|X-\Xh_n\|_\infty\geq \bar{c}\,\sqrt{\ln(n+1)/n} \bigr) \geq \bar{\gamma}.
	\end{align*}
	In particular, for $\hat{c}=\bar{c}\cdot\bar{\gamma}\in(0,\infty)$ we have for all $n\in\N$ that
	\begin{align*}
		\inf_{\Xh_n\in\Ac_n(C([0,T]),X(0),W)} {\left\{ \EE\bigl[ \|X-\Xh_n\|_\infty \bigr] \right\}}
			\geq \hat{c}\cdot \sqrt{\ln(n+1)/n}.
	\end{align*}
\end{thm}

\begin{proof}
	The proof of Theorem~\ref{thm:sup-local} is almost identical to the proof of Theorem~\ref{thm:ft-local}.
	Instead of Proposition~\ref{prop:ft-local} we rely on Proposition~\ref{prop:sup-local}.
\end{proof}

\subsection{$L_p$-Approximation}\label{sec:global-Lp}

In this section we consider approximation with respect to the $L_p$-distance
on the time interval $[0,T]$ for $p\in[1,\infty)$.
The following proposition extends mean error bounds from \citet*{HMGR00-1,MG02_habil} to probability bounds for the error of $\Xh_n$ from $\Aa_n(L_p([0,T]),W)$.

\begin{prop}[Probability bounds on the $L_p$-error of adaptive methods I]\label{prop:Lp-global}
	Assume the setting in Section~\ref{sec:setting} with $\EE{\left[|X(0)|^4\right]}<\infty$.
	Assume further that
	\begin{enumerate}[\hspace{.5cm}({C}1)]\addtolength{\itemsep}{0.25\baselineskip}
		\item\label{cond:regularity-Lp} $\forall\,i\in\{0,1\},\,j\in\{0,1,2\}\colon\,$the partial derivatives $a^{(i,j)}, b^{(i,j)}$
			exist on $[0,T]\times\R$ and are continuous and bounded,
		\item\label{cond:nondeg-Lp} $\inf_{(t,x)\in [0,T]\times\R} |b(t,x)|>0$.
	\end{enumerate}
	Then for every $\varepsilon\in (0,1)$ there exists $c\in(0,\infty)$ such that
	for all $p\in[1,\infty)$, for all $n\in\N$, and for all $\Xh_n\in\Aa_n(L_p([0,T]),W)$ we have
	\begin{align*}
		\PP\bigl( \|X-\Xh_n\|_p\geq c/\sqrt{n} \bigr) \geq 1-\varepsilon.
	\end{align*}
\end{prop}

\begin{proof}
	Due to monotonicity in $p$ we may assume $p=1$. Moreover, due to a scaling in time we may assume $T=1$.
	
	For $k\in\N$ we put $t_{\ell}=\ell/k$ for $\ell=0,\dots,k$ and we define a
	continuous-time Milstein scheme $\Xm_k=\bigl(\Xm_k(t)\bigr)_{t\in [0,1]}$ by $\Xm_k(0)=X(0)$
	and for $\ell=0,\dots,k-1$ and $t\in(t_\ell,t_{\ell+1}]$ by
	\begin{multline*}
		\Xm_k(t) = \Xm_k(t_\ell) + a\bigl( t_\ell,\Xm_k(t_\ell) \bigr)\cdot (t-t_\ell)
			+ b\bigl( t_\ell,\Xm_k(t_\ell) \bigr)\cdot \bigl( W(t)-W(t_\ell) \bigr) \\
		+ 1/2\cdot\bigl(bb^{(0,1)}\bigr)\bigl(t_\ell,\Xm_k(t_\ell)\bigr)
				\cdot\bigl(\bigl(W(t)-W(t_\ell)\bigr)^2-(t-t_\ell)\bigr).
	\end{multline*}
	By (C\ref{cond:regularity-Lp}) and $\EE{\left[|X(0)|^4\right]}<\infty$ we have
	\begin{align}\label{rate1}
		\exists\,c_1\in(0,\infty) \enspace \forall\,k\in\N \colon\,
			\sup_{t\in [0,1]}\EE\bigl[|X(t)-\Xm_k(t)|\bigr]  \leq c_1\cdot k^{-1},
	\end{align}
	see, e.g., \citet*[Theorem~4]{HMGR00-1}.
	For $k\in\N$ and $\ell=0,\dots,k-1$ we define the process $U_k^{(\ell)}=\bigl(U_k^{(\ell)}(t)\bigr)_{t\in [t_\ell,t_{\ell+1}]}$ by
	\begin{align*}
		U_k^{(\ell)}(t) 
		&= \Xm_k(t)-1/2\cdot\bigl(bb^{(0,1)}\bigr)\bigl(t_\ell,\Xm_k(t_\ell)\bigr)
			\cdot\bigl(\bigl(W(t)-W(t_\ell)\bigr)^2-(t-t_\ell)\bigr).
	\end{align*}
	Furthermore, for $k\in\N$ we define an auxiliary scheme $\Xmaux_k = \bigl(\Xmaux_k(t)\bigr)_{t\in [0,1]}$
	by $\Xmaux_k(0)=X(0)$ and for $\ell=0,\dots,k-1$ and $t\in (t_\ell,t_{\ell+1}]$ by
	\begin{align*}
		\Xmaux_k(t) = U_k^{(\ell)}(t).
	\end{align*}
	By (C\ref{cond:regularity-Lp}) we have
	\begin{align*}
		c_2 = \sup_{(t,x)\in [0,1]\times\R} |bb^{(0,1)}(t,x)|\in [0,\infty).
	\end{align*}
	Hence we get from \eqref{rate1} for $c_3=c_1+c_2\in(0,\infty)$ and all $k\in\N$ that
	\begin{equation}\label{rate1aux}
\begin{aligned}
			 & \EE\bigl[\|X-\Xmaux_k\|_1\bigr] \\ & \qquad \leq \sup_{t\in [0,1]}\EE\bigl[|X(t)-\Xmaux_k(t)|\bigr]
			 = \sup_{\substack{ \ell\in\{0,\dots,k-1\}\\ t\in(t_\ell,t_{\ell+1}] }} \EE\bigl[|X(t)-U_k^{(\ell)}(t)|\bigr] \\
			&\qquad \leq \sup_{\substack{ \ell\in\{0,\dots,k-1\}\\ t\in(t_\ell,t_{\ell+1}] }} \EE\bigl[|X(t)-\Xm_k(t)|\bigr] \\
			&\qquad\qquad+ \sup_{\substack{ \ell\in\{0,\dots,k-1\}\\ t\in(t_\ell,t_{\ell+1}] }}
				\EE\Bigl[\bigl| 1/2\cdot\bigl(bb^{(0,1)}\bigr)\bigl(t_\ell,\Xm_k(t_\ell)\bigr)
				\cdot\bigl((W(t)-W(t_\ell))^2-(t-t_\ell)\bigr)\bigr|\Bigr] \\
			&\qquad\leq c_1\cdot k^{-1} + \frac{c_2}{2}\cdot 2\cdot k^{-1} = c_3\cdot k^{-1}.
		\end{aligned}
	\end{equation}
	
	Let $\varepsilon\in(0,1)$ and
	\begin{align*}
	 \delta=\inf_{(t,x)\in [0,1]\times\R} |b(t,x)|\in (0,\infty),
	\end{align*}
	see (C\ref{cond:nondeg-Lp}), and choose $c_4\in(0,\infty)$ according to Lemma~\ref{lem:normal3}.

	Let $k\in 2\N$ and consider an approximation $\Xh_k\in\Aa_{k/2}(L_1([0,1]), W)$.
	Observe that there exist random variables $\tau_0,\dots,\tau_{3k/2}\colon \Omega\to [0,\infty)$ with
	\begin{enumerate}[(i)]\addtolength{\itemsep}{0.25\baselineskip}
		\item $\tau_\ell$ is measurable with respect to $\sigma\bigl(\Fc_0,W(\tau_0),\dots, W(\tau_{\ell-1})\bigr)$ for all $\ell=0,\dots,3k/2$,
		\item $\tau_\ell=t_\ell$ for all $\ell=0,\dots,k$,
		\item $\Xh_k$ is measurable with respect to the $\sigma$-algebra $\A=\sigma\bigl(\Fc_0,W(\tau_0),\dots, W(\tau_{3k/2})\bigr)$.
	\end{enumerate}
	For $\ell\in\{0,\dots,k-1\}$ define $d_\ell = \#\{i\in\{0,\dots,3k/2 \}\colon\,\tau_i\in(t_\ell,t_{\ell+1})\}$.
	Let $\ell_1,\dots,\ell_{k/2}\in\{0,\dots,k-1\}$ such that
	$\ell_1<\dots<\ell_{k/2}$, $d_{\ell_i}=0$ for all $i=1,\dots,k/2$,
	and $d_\ell>0$ for all $\ell\in\{0,\dots,\ell_{k/2}\}\setminus\{\ell_1,\dots,\ell_{k/2}\}$.
	We then have
	\begin{align*}
		\|\Xh_k-\Xmaux_k\|_1
			&\geq \sum_{i=1}^{k/2} \int_{t_{\ell_i}}^{t_{\ell_i+1}} |\Xh_k(t)-\Xmaux_k(t)| \,\dd t\\
		&= \sum_{i=1}^{k/2} \int_{t_{\ell_i}}^{t_{\ell_i+1}} |\Xh_k(t)-U_k^{(\ell_i)}(t)| \,\dd t\\
		&= \sum_{i=1}^{k/2}\int_{0}^{1/k} |\Xh_k(t+t_{\ell_i})-U_k^{(\ell_i)}(t+t_{\ell_i})| \,\dd t
	\end{align*}
	and thus
	\begin{align}\label{eq100}
		\begin{aligned}
			&\PP\bigl( \|\Xh_k-\Xmaux_k\|_1\geq c_4/\sqrt{k/2} \,\big|\, \A \bigr) \\
			&\qquad\geq \PP\Bigl( \sum_{i=1}^{k/2}\int_{0}^{1/k}
				|\Xh_k(t+t_{\ell_i})-U_k^{(\ell_i)}(t+t_{\ell_i})|\,\dd t \geq c_4/\sqrt{k/2}\,\Big|\, \A \Bigr).
		\end{aligned}
	\end{align}
	Conditioned on $\A$, the values of $\bigl(\Xm_k(t_\ell)\bigr)_{\ell\in \{0,\dots,k\}}$,
	$W(t_0),\dots,W(t_k)$, $\ell_1,\dots,\ell_{k/2}$, and $\Xh_k$ are fixed
	and the processes $\bigl(B_1(t)\bigr)_{t\in [0,1/k]},\dots,\bigl(B_{k/2}(t)\bigr)_{t\in [0,1/k]}$, given by
	\begin{align*}
		B_i(t) = W(t+t_{\ell_i}) - \bigl( (1-k\cdot t)\cdot W(t_{\ell_i}) + k\cdot t\cdot W(t_{\ell_i+1}) \bigr),
	\end{align*}
	are independent Brownian bridges (from $0$ to $0$), cf.~\citet*[Lemmas~1,2]{yaroslavtseva2017}.
	Furthermore, for $i\in\{1,\dots,k/2\}$ and $t\in [0,1/k]$ it holds that
	\begin{align*}
		U_k^{(\ell_i)}&(t+t_{\ell_i}) - \Xh_k(t+t_{\ell_i}) \\
		&= b\bigl(t_{\ell_i},\Xm_k(t_{\ell_i})\bigr) \cdot B_i(t)
			+ \Bigl[ \Xm_k(t_{\ell_i}) + a\bigl(t_{\ell_i},\Xm_k(t_{\ell_i})\bigr)\cdot t\\
		&\qquad+ b\bigl(t_{\ell_i},\Xm_k(t_{\ell_i})\bigr) \cdot
			\Bigl( (-k\cdot t)\cdot W(t_{\ell_i}) + k\cdot t\cdot W(t_{\ell_i+1}) \Bigr)-\Xh_k(t+t_{\ell_i}) \Bigr].
	\end{align*}
	Hence we get from Lemma~\ref{lem:normal3} that
	\begin{align*}
		\PP\Bigl( \sum_{i=1}^{k/2}\int_{0}^{1/k} |\Xh_k(t+t_{\ell_i})-U_k^{(\ell_i)}(t+t_{\ell_i})|\,\dd t
			\geq c_4/\sqrt{k/2} \,\Big|\,\A \Bigr) \geq 1-\varepsilon.
	\end{align*}
	Combining this with \eqref{eq100} yields
	\begin{align}\label{eqwert}
	\PP\bigl( \|\Xh_k-\Xmaux_k\|_1\geq c_4/\sqrt{k/2} \bigr) \geq 1-\varepsilon.
	\end{align}

	Put $c_5=c_4/\sqrt{2}\in(0,\infty)$. Using the Markov inequality we derive from~\eqref{rate1aux}
	that for all $k\in\N$ we have
	\begin{align}\label{eq:new101}
		\PP\bigl( \|X-\Xmaux_k\|_1\geq c_5/\sqrt{k} \bigr)
			\leq \sqrt{k}/c_5\cdot c_3\cdot k^{-1} = c_3/c_5\cdot k^{-1/2}.
	\end{align}
	From \eqref{eqwert} and \eqref{eq:new101} we get that for all $k\in 2\N\cap[(c_3/(c_5\varepsilon))^2,\infty)$
	and $\Xh_k\in\Aa_{k/2}(L_1([0,1]), W)$ it holds that
	\begin{align*}
		\PP\bigl( \|X-\Xh_k\|_1\geq c_5/\sqrt{k} \bigr)
			&\geq \PP\bigl( \bigl\{\|\Xmaux_{k}-\Xh_k\|_1\geq 2c_5/\sqrt{k}\bigr\}
				\cap \bigl\{\|X-\Xmaux_{k}\|_1< c_5/\sqrt{k}\bigr\} \bigr) \\
		&\geq \PP\bigl( \|\Xmaux_{k}-\Xh_k\|_1\geq 2c_5/\sqrt{k} \bigr) - \PP\bigl( \|X-\Xmaux_{k}\|_1\geq c_5/\sqrt{k} \bigr)\\
		&\geq 1-\varepsilon - c_3/c_5\cdot k^{-1/2} \geq 1-2\varepsilon,
	\end{align*}
	which completes the proof.
\end{proof}

Combining the localization technique from Appendix~\ref{sec:comparison-result}
with Proposition~\ref{prop:Lp-global} leads to the following result.

\begin{prop}[Probability bounds on the $L_p$-error of adaptive methods II]\label{prop:Lp-local}
	Assume the setting in Section~\ref{sec:setting}.
	Let $t_0\in[0,T)$, $T_0\in (t_0,T]$ and let $\emptyset\neq I\subseteq\R$ be
	an open interval such that
	\begin{enumerate}[\hspace{.5cm}({C}1*)]\addtolength{\itemsep}{0.25\baselineskip}
		\item $\forall\,i\in\{0,1\},\,j\in\{0,1,2\}\colon\,$the partial derivatives $a^{(i,j)}, b^{(i,j)}$
			exist on $[t_0,T_0]\times I$ and are continuous,
		\item $\forall\,(t,x)\in [t_0,T_0]\times I\colon\, b(t,x)\neq 0$,
		\item $\PP(X(t_0)\in I)>0$.
	\end{enumerate}
	Then there exist constants $c,\gamma\in(0,\infty)$ such that
	for all $p\in[1,\infty)$, for all $n\in\N$, and for all $\Xh_n\in\Aa_n(L_p([0,T]),W)$ we have
	\begin{align*}
		\PP\bigl( \|X-\Xh_n\|_p \geq c/\sqrt{n} \bigr) \geq \gamma.
	\end{align*}
\end{prop}

\begin{proof}
	Since the error is measured globally in time, we may assume $T_0=T$. Now, the reasoning is almost identical to the reasoning in the proof of Proposition~\ref{prop:ft-local}.
	Instead of Proposition~\ref{prop:ft-global} we rely on Proposition~\ref{prop:Lp-global}.
\end{proof}

Proposition~\ref{prop:Lp-local} provides a lower error bound for methods from the class $\Aa_n(L_p([0,T]),W)$, i.e., for adaptive methods that are based on $n$ evaluations of $W$. 
We now extend this result to the class $\Ac_n(L_p([0,T]),X(0),W)$ of adaptive methods that may us $n$ evaluations of $W$, on average. The following theorem is our main result for $L_p$-Approximation with $p\in [1,\infty)$.

\begin{thm}[Lower error bound for $L_p$-approximation]\label{thm:Lp-local}
Assume the setting in Section~\ref{sec:setting}.
	Let $t_0\in[0,T)$, $T_0\in (t_0,T]$ and let $\emptyset\neq I\subseteq\R$ be
	an open interval such that	the conditions~(C\ref{cond1}*),~(C\ref{cond2}*), and~(C\ref{cond3}*) from Proposition~\ref{prop:Lp-local} are satisfied.
	Then there exist constants $\bar{c},\bar{\gamma}\in(0,\infty)$ such that
	for all $p\in[1,\infty)$, for all $n\in\N$, and for all $\Xh_n\in\Ac_n(L_p([0,T]),X(0),W)$ we have
	\begin{align*}
		\PP\bigl( \|X-\Xh_n\|_p\geq \bar{c}/\sqrt{n} \bigr) \geq \bar{\gamma}.
	\end{align*}
	In particular, for $\hat{c}=\bar{c}\cdot\bar{\gamma}\in(0,\infty)$ we have for all $p\in[1,\infty)$ and for all $n\in\N$ that
	\begin{align*}
		\inf_{\Xh_n\in\Ac_n(L_p([0,T]),X(0),W)} {\left\{ \EE\bigl[ \|X-\Xh_n\|_p \bigr] \right\}}
			\geq \hat{c}\cdot n^{-1/2}.
	\end{align*}
\end{thm}

\begin{proof}
	The proof of Theorem~\ref{thm:Lp-local} is almost identical to the proof of Theorem~\ref{thm:ft-local}.
	Instead of Proposition~\ref{prop:ft-local} we rely on Proposition~\ref{prop:Lp-local}.
\end{proof}

\subsection{Maximum Pointwise Error}\label{sec:global-sup-outside}

The following theorem is a consequence of
Theorem~\ref{thm:Lp-local} and Fubini's theorem.

\begin{thm}[Lower  bound for  the maximum pointwise  error]\label{thm:uniform-pointwise}
Assume the setting in Section~\ref{sec:setting}.
	Let $t_0\in[0,T)$, $T_0\in (t_0,T]$ and let $\emptyset\neq I\subseteq\R$ be
	an open interval such that	the conditions~(C\ref{cond1}*),~(C\ref{cond2}*), and~(C\ref{cond3}*) from Proposition~\ref{prop:Lp-local} are satisfied.
	Then there exists a constant $\hat{c}\in(0,\infty)$ such that
	for all $n\in\N$ we have
	\begin{align*}
		\inf_{\Xh_n\in\Ac_n(C([0,T]),X(0),W)}\;\biggl\{ \sup_{t\in[0,T]} \EE\bigl[|X(t)-\Xh_n(t)|\bigr] \biggr\}
			\geq \hat{c}\cdot n^{-1/2}.
	\end{align*}
\end{thm}

\begin{proof}
	Let $n\in\N$ and consider an approximation $\Xh_n\in\Ac_n(C([0,T]),X(0),W)$.
	Fubini's theorem shows
	\[
		\EE\bigl[ \|X-\Xh_n\|_1 \bigr] 
		= \int_0^T \EE\bigl[ |X(t)-\Xh_n(t)| \bigr] \,\dd t
			\leq T\cdot \sup_{t\in[0,T]} \EE\bigl[|X(t)-\Xh_n(t)|\bigr].
	\]
	Applying Theorem~\ref{thm:Lp-local} for $p=1$ completes the proof.
\end{proof}

\section{Examples}\label{sec:examples}

We study Cox-Ingersoll-Ross processes in Section~\ref{sec:cir},
equations with superlinearly growing coefficients in Section~\ref{sec:superlinear},
and equations with discontinuous coefficients in Section~\ref{sec:discontinuous}.
Throughout this section we assume the setting in Section~\ref{sec:setting}.

\subsection{Cox-Ingersoll-Ross Processes}\label{sec:cir}

Cox-Ingersoll-Ross processes are widely used in mathematical finance, e.g.,
in the Cox-Ingersoll-Ross model for short term interest rates or as the
instantaneous variance in the Heston model. Such processes are described by the SDE
\begin{align}\label{eq:cir}
	\dd X(t) = (\delta-\beta X(t))\,\dd t +\sigma\sqrt{|X(t)|}\,\dd W(t), \qquad X(0)=x_0,
\end{align}
with initial value $x_0\in (0,\infty)$ and parameters $\delta,\sigma\in (0,\infty)$,
$\beta\in [0,\infty)$. In this case the coefficients $a,b$ are given by
\begin{align*}
	a(t,x) = a(x) = \delta-\beta\cdot x
		\qquad\text{and}\qquad
	b(t,x) = b(x) = \sigma\cdot\sqrt{|x|}
\end{align*}
for every $(t,x)\in [0,T]\times\R$.
It is well-known that strong existence and pathwise uniqueness hold for the SDE~\eqref{eq:cir}.
Furthermore, one has
\begin{align*}
\PP\!\left(\forall\,t\in [0,\infty)\colon X(t)\geq 0\right)=1.
\end{align*}
Due to a simple scaling
we may restrict ourselves to the case
\begin{align*}
	\sigma = 2,
\end{align*}
which we assume in the following.

In the context of strong approximation, it turns out that
$\delta$ is the crucial parameter with respect to the rate of convergence,
see, e.g.,
Corollary~\ref{cor:CIR-global}(\ref{cor15partiv})  below.

\subsubsection{Strong Approximation at a Single Point}

Note that the coefficients $a,b$ of the SDE~\eqref{eq:cir} are infinitely differentiable on $(0,\infty)$. For $x\in(0,\infty)$ we obtain
\[
	\bigl( a'b - a\:\!b' - \tfrac{1}{2} b^2b'' \bigr)(x)
		= -\beta\sqrt{x} + \frac{1-\delta}{\sqrt{x}}.
\]
In particular, it holds that
\begin{equation}\label{eq:cond1}
\begin{aligned}
	\exists\,\emptyset\neq I\subseteq (0,\infty)\text{ open interval}\enspace
		\forall\,x\in I\colon\, \bigl( a'b - a\:\!b' - \tfrac{1}{2} b^2b'' \bigr)(x) \neq 0 & \\
		\Leftrightarrow \hspace*{5cm}& \\
		 \hspace*{2cm} \delta\neq 1\text{ or }\beta\neq0. \hspace*{4cm} &
\end{aligned}
\end{equation}
The marginal distributions of the SDE~\eqref{eq:cir} are known explicitly in terms of Lebesgue-densities, see, e.g., \citet*[Equation~(4) and Section~A.2]{yor}.
The latter results immediately yield
\begin{align}\label{eq:density}
	\forall\,\emptyset\neq I\subseteq (0,\infty)\text{ open interval}\enspace
		\forall\,t\in(0,\infty)\colon\, \PP( X(t)\in I ) > 0.
\end{align}
In case of $\delta\neq 1$ or $\beta\neq0$ we get from \eqref{eq:cond1} and \eqref{eq:density}
that the assumptions of Theorem~\ref{thm:ft-local} are fulfilled for every $t_0\in(0,T)$,
which yields the lower bound stated in part~\eqref{cor14parti} of the following corollary.
Combining this lower bound with the upper bound of \citet*[Theorem~2]{Alf13}
yields a sharp result for $\delta>4$, see part~\eqref{cor14partii}.

\begin{cor}[CIR processes, pointwise approximation]\label{cor:CIR-pointwise}\ 
	\begin{enumerate}[(i)]\addtolength{\itemsep}{0.25\baselineskip}
	\item\label{cor14parti} If $\delta\neq 1$ or $\beta\neq0$, then there exists a constant $c\in(0,\infty)$
		such that for all $n\in\N$ we have
		\begin{align*}
			\inf_{\Xh_n(T)\in\Ac_n(\R,X(0),W)} {\left\{ \EE\Bigl[ \bigl|X(T)-\Xh_n(T)\bigr| \Bigr] \right\}}
				\geq c \cdot n^{-1}.
		\end{align*}
	\item\label{cor14partii} If $\delta>4$, then there exist constants $c,C\in(0,\infty)$ such that
		for all $n\in\N$ we have
		\begin{align*}
			c\cdot n^{-1} &\leq \inf_{\Xh_n(T)\in\Ac_n(\R,X(0),W)}
				{\left\{ \EE\Bigl[ \bigl|X(T)-\Xh_n(T)\bigr| \Bigr] \right\}}\\
			&\leq \inf_{\Xh_n(T)\in \Ac^\mathrm{eq}_n(\R,X(0),W)}
				{\left\{ \EE \bigl[\vert X(T)-\Xh_n(T)\vert\bigr] \right\}} \leq C\cdot n^{-1}.
		\end{align*}
	\end{enumerate}
\end{cor}

Let us stress that Corollary~\ref{cor:CIR-pointwise}(\ref{cor14partii}) shows that, in case of $\delta >4$, adaptive algorithms
are not superior (up to multiplicative constants) to algorithms that are based on fixed equidistant grids,
similar to the case of SDEs with coefficients that satisfy global Lipschitz assumptions.
Moreover, we recover the usual optimal
convergence rate of~$1$.

If $\delta<1$ then combining the upper bound of \citet*[Theorem~2]{HH16b} with the lower bound of \citet*[Theorem~1]{HJ17}
shows that there exist constants $c,C\in(0,\infty)$ such that
for all $n\in\N$ it holds that
\begin{align}\label{eq:CIR-equi}
	c \cdot n^{-\delta/2} \leq \inf_{\Xh_n(T)\in \Ac^\mathrm{eq}_n(\R,X(0),W)}
		{\left\{ \EE \bigl[\vert X(T)-\Xh_n(T)\vert\bigr] \right\}} \leq C\cdot n^{-\delta/2}.
\end{align}
Hence, for this range of values of $\delta$ the lower bound  in Corollary~\ref{cor:CIR-pointwise}(\ref{cor14parti}) cannot be attained by 
algorithms that are based on equidistant grids and it  
is unclear, up to now, whether the bound in Corollary~\ref{cor:CIR-pointwise}(\ref{cor14parti}) is sharp, i.e., whether  adaptive algorithms can achieve the convergence rate $c/n$ in terms of the average number $n$ of evaluations of the driving Brownian motion that are used.
Note, however, that the result \eqref{eq:adap} on the power of adaptive methods in the case $\delta=1$ and $\beta=0$ provides a positive indication in that sense.

Combining Corollary~\ref{cor:CIR-pointwise}(\ref{cor14partii}) with \eqref{eq:CIR-equi} shows
that the optimal convergence rate for algorithms that are based on equidistant grids
is exactly $\min(\delta/2,1)$ if $\delta\in(0,1)\cup(4,\infty)$.
For $\delta\in(1,4)$, no matching upper and lower bounds are known,
see \citet*[Figure~1]{HH16a}.
Nevertheless, we expect $\min(\delta/2,1)$ to be the best possible convergence rate
for algorithms that are based on equidistant grids for all $\delta\in(0,\infty)$.

Finally, we consider the case of
\begin{align*}
	\delta=1 \qquad\text{and}\qquad \beta=0,
\end{align*}
i.e., the solution of the SDE~\eqref{eq:cir} is a one-dimensional squared Bessel process.
In this case none of the results of Section~\ref{sec:ft} is applicable, see~\eqref{eq:cond1}.
Indeed, in \citet*[Theorem~4]{HH16a}
it is shown that adaptive methods are able to
achieve any polynomial convergence order and hence are not restricted to rate~$1$.
See also \citet*{CHH17}.
More precisely, \citet*[Theorem~4]{HH16a} shows that for all $r\in[1,\infty)$
there exists a constant $C_r\in(0,\infty)$ such that for all $n\in\N$ it holds that
\begin{align}\label{eq:adap}
\inf_{\Xh_n(T)\in \Ac_n(\R,X(0),W)}
		{\left\{ \EE \bigl[\vert X(T)-\Xh_n(T)\vert\bigr] \right\}}
		\leq C_r\cdot n^{-r}.
\end{align}
On the other hand, \citet*[Corollary~1]{HH16a} shows that the best algorithm based on equidistant nodes
converges at rate $1/2$, i.e., there exist constants $c,C\in(0,\infty)$ such that for all $n\in\N$ we have
\begin{align*}
	c \cdot n^{-1/2}\leq \inf_{\Xh_n(T)\in \Ac^\mathrm{eq}_n(\R,X(0),W)}
		{\left\{ \EE \bigl[\vert X(T)-\Xh_n(T)\vert\bigr] \right\}} \leq C\cdot n^{-1/2}.
\end{align*}

\subsubsection{Strong Approximation Globally in Time}

We now turn to lower error bounds that are based on Theorem~\ref{thm:sup-local}
and Theorem~\ref{thm:Lp-local}.
Moreover, these lower bounds are complemented by upper error bounds from the literature.

\begin{cor}[CIR processes, global approximation]\label{cor:CIR-global}\ 
	\begin{enumerate}[(i)]\addtolength{\itemsep}{0.25\baselineskip}
	\item\label{cor15parti} For all $\delta\in (0,\infty)$ and $\beta\in [0,\infty)$ there exists a constant $c\in(0,\infty)$ such that for all $n\in\N$ we have
		\[
			\inf_{\Xh_n\in\Ac_n(C([0,T]),X(0),W)} {\left\{ \EE\bigl[ \|X-\Xh_n\|_\infty \bigr] \right\}}
				\geq c\cdot \sqrt{\ln(n+1)/n}
		\]
		and
		\[
			\inf_{\Xh_n\in\Ac_n(L_1([0,T]),X(0),W)} {\left\{ \EE\bigl[ \|X-\Xh_n\|_1 \bigr] \right\}}
				\geq c\cdot n^{-1/2}.
		\]
		\item\label{cor15partii} If $\delta>2$ and $\beta>0$ or if $\delta=1$, then there exist constants
		$c,C\in (0,\infty)$ such that for all $n\in\N$ we have
		\begin{align*}
			c\cdot \sqrt{\ln(n+1)/n} &\leq \inf_{\Xh_n\in\Ac_n(C([0,T]),X(0),W)}
				{\left\{ \EE\bigl[ \|X-\Xh_n\|_\infty \bigr] \right\}} \\
			&\leq \inf_{\Xh_n\in\Ac^\mathrm{eq}_n(C([0,T]),X(0),W)}
				{\left\{ \EE\bigl[ \|X-\Xh_n\|_\infty \bigr] \right\}} \leq C\cdot \sqrt{\ln(n+1)/n}.
		\end{align*}	
		\item\label{cor15partiii} If $\delta>1$  then there exist constants
		$c,C\in (0,\infty)$ such that for all $n\in\N$ we have
		\begin{align*}
			c\cdot n^{-1/2} &\leq \inf_{\Xh_n\in\Ac_n(L_1([0,T]),X(0),W)}
				{\left\{ \EE\bigl[ \|X-\Xh_n\|_1 \bigr] \right\}} \\
			&\leq \inf_{\Xh_n\in\Ac^\mathrm{eq}_n(L_1([0,T]),X(0),W)}
				{\left\{ \EE\bigl[ \|X-\Xh_n\|_1 \bigr] \right\}} \leq C\cdot n^{-1/2}.
		\end{align*}
	\item\label{cor15partiv} For all $\delta\in (0,\infty)$ and $\beta\in [0,\infty)$ there exist constants $c,C\in (0,\infty)$ such that for all $n\in\N$ we have
\begin{align*}
			c\cdot n^{-\min(1/2,\delta/2)} & \leq \inf_{\Xh_n\in\Ac_n^\mathrm{eq}(C([0,T]),X(0),W)}
				\biggl\{ \sup_{t\in [0,T]}\EE\bigl[|X(t)-\Xh_n(t)| \bigr] \biggr\} \\
		&	\leq C\cdot \bigl( 1+\ind_{\{1\}}(\delta)\cdot\ind_{\R\setminus\{0\}}(\beta)\cdot\sqrt{\ln(n)} \bigr)
				\cdot n^{-\min(1/2,\delta/2)}.
		\end{align*}
	\end{enumerate}
\end{cor}

\begin{proof}
Similar to the reasoning in the previous section we conclude that the assumptions
of Theorem~\ref{thm:sup-local} and Theorem~\ref{thm:Lp-local} are fulfilled
without any restriction on the parameters $\delta$ and $\beta$. This shows part~\eqref{cor15parti}.
	
The lower bound in~\eqref{cor15partii} clearly follows from the first statement in~\eqref{cor15parti}.
The upper bound in~\eqref{cor15partii} is due to
	\citet*[Theorem~1.1]{DNS12} if $\delta>2$ and $\beta>0$ and due to
	\citet*[Remark~7]{HH16a}(with a linearly interpolated version of the corresponding numerical scheme) if $\delta=1$.
	
	The lower bound in~\eqref{cor15partiii} clearly follows from the second statement in~\eqref{cor15parti}.
	The upper bound in~\eqref{cor15partiii} is a consequence of the upper bound in~\eqref{cor15partiv}.
	
	The lower bound in~\eqref{cor15partiv} follows from the second statement in~\eqref{cor15parti} and \eqref{eq:CIR-equi}. The upper bound in~\eqref{cor15partiv} follows from \citet*[Theorem~2]{HH16b} (with a linearly interpolated version of the corresponding numerical scheme) if $\delta\neq 1$,
	from the upper bound in \eqref{cor15partii} if $\delta=1$ and $\beta\neq 0$,
	and from \citet*[Corollary~1]{HH16a} (with a linearly interpolated version of the corresponding numerical scheme) if $\delta=1$ and $\beta=0$.
\end{proof}

Concerning Corollary~\ref{cor:CIR-global}(\ref{cor15partii}) we add that no matching upper and lower bounds are known
for $\delta\in(0,2)\setminus\{1\}$, cf.~\citet*[Figure~1]{HH16b}.
Concerning Corollary~\ref{cor:CIR-global}(\ref{cor15partiii}) we add that no matching upper and lower bounds are known for $\delta\in(0,1)$.

We stress that Corollary~\ref{cor:CIR-global}(\ref{cor15partiv}) is the first result in the literature that provides matching upper and lower bounds (up to logarithmic terms)
for a particular error criterion without any restriction on the parameters.

\subsection{SDEs with Superlinearly Growing Coefficients}\label{sec:superlinear}

For simplicity, we assume throughout this section that
\begin{align*}
	a(t,x)=a(x) \qquad\text{and}\qquad b(t,x)=b(x) \qquad\text{are polynomials.}
\end{align*}
Furthermore, for a polynomial $h\colon\R\to\R$ we put
\begin{align*}
	\zeros(h)=\{x\in\R\colon\, h(x)=0\}.
\end{align*}
The following result is an immediate consequence of Theorem~\ref{thm:ft-local} as well as Theorem~\ref{thm:sup-local} and Theorem~\ref{thm:Lp-local}.

\begin{cor}[Superlinearly growing coefficients]\label{cor:superlinear}\ 
	\begin{enumerate}[(i)]\addtolength{\itemsep}{0.25\baselineskip}
	\item Assume that
		\begin{align*}
			\PP\!\left( X(0)\notin\bigl( \zeros(b)\cup\zeros(a'b-a\:\!b'-\tfrac{1}{2}b^2b'') \bigr)\right) > 0.
		\end{align*}
		Then there exists a constant $c\in(0,\infty)$ such that for all $n\in\N$ we have
		\begin{align*}
			\inf_{\Xh_n(T)\in\Ac_n(\R,X(0),W)} {\left\{ \EE\Bigl[ \bigl|X(T)-\Xh_n(T)\bigr| \Bigr] \right\}}
				\geq c \cdot n^{-1}.
		\end{align*}
	\item Assume that
		\begin{align*}
			\PP \bigl( X(0)\notin \zeros(b) \bigr) > 0.
		\end{align*}
		Then there exists a constant $c\in(0,\infty)$ such that for all $n\in\N$ we have
		\[
			\inf_{\Xh_n\in\Ac_n(C([0,T]),X(0),W)} {\left\{ \EE\bigl[ \|X-\Xh_n\|_\infty \bigr] \right\}}
				\geq c\cdot \sqrt{\ln(n+1)/n}
		\]
		and
	\[
			\inf_{\Xh_n\in\Ac_n(L_1([0,T]),X(0),W)} {\left\{ \EE\bigl[ \|X-\Xh_n\|_1 \bigr] \right\}}
				\geq c\cdot n^{-1/2}.
		\]
		\end{enumerate}
\end{cor}

There are matching upper error bounds  under certain monotone conditions on the coefficients. These bounds are  achieved by  tamed or projected or implicit versions of the Euler scheme in case of the mean $L_p$-error and the mean $L_\infty$-error, see~\citet*{HJK12, Hutzenthaler-Jentzen-Kloeden2013,sabanis:2016,BIK16}, and of the Milstein scheme in case of the error at a single time, see~\citet*{WG13,Kumar-Sabanis2016,BIK17}. In all of these cases, the corresponding methods are non-adaptive and based on a fixed equidistant grid, and therefore, adaptive algorithms are not superior to non-adaptive ones.
For an example see equation \eqref{sde2} and the subsequent discussion in the introduction.

\subsection{SDEs with Discontinuous Coefficients}\label{sec:discontinuous}

As an illustrating example we consider the SDE
\begin{align}\label{eq:discontinuous}
	\dd X(t) = \sgn(X(t))\cdot (1+X(t))\,\dd t + \dd W(t), \qquad X(0)=x_0,
\end{align}
with initial value $x_0\in\R$,
where $\sgn(x)=1$ for $x\in [0,\infty)$ and $\sgn(x)=-1$ for $x\in (-\infty,0)$.
Here, the coefficients $a,b$ are given by
\begin{align*}
	a(t,x)=a(x)=\sgn(x)\cdot (1+x) \qquad\text{and}\qquad b(t,x)=b(x)=1
\end{align*}
for every $(t,x)\in [0,T]\times\R$. Strong existence and pathwise uniqueness for
the SDE~\eqref{eq:discontinuous} follows from, e.g., \citet*[Theorem~4]{Zv74}.
Clearly, the coefficients $a,b$ of the SDE~\eqref{eq:discontinuous} are infinitely often differentiable
on $\R\setminus\{0\}$.
Moreover, for every $x\in\R\setminus\{0\}$ it holds that
\begin{align*}
	\bigl( a'b - a\:\!b' - \tfrac{1}{2} b^2b'' \bigr)(x)  = \sgn(x)
\end{align*}
and it is easy to see that
\begin{align*}
	\PP(\forall\,t\in[0,T/2]\colon\, X(t)=0) \neq 1.
\end{align*}
Hence the assumptions of Theorems~\ref{thm:ft-local},~\ref{thm:sup-local},~\ref{thm:Lp-local} are satisfied.
Combining this with an upper bound of \citet*[Theorem~3.1]{LS16}
(with a linearly interpolated version of the corresponding numerical scheme)
shows that there exist constants $c,C\in(0,\infty)$ such that for all $n\in\N$ we have
\begin{align*}
	 \inf_{\Xh_n\in\Ac_n(\R,X(0),W)}
		\EE\bigl[|X(T)-\Xh_n(T)| \bigr] & \geq c\cdot n^{-1},\\
	\inf_{\Xh_n\in\Ac_n(C([0,T]),X(0),W)}  \EE\bigl[ \|X-\Xh_n\|_\infty \bigr] 
				& \geq c\cdot \sqrt{\ln(n+1)/n}	
\end{align*}
as well as
\begin{align*}
	c \cdot n^{-1/2} &\leq \inf_{\Xh_n\in\Ac_n(L_1([0,T]),X(0),W)}
		\EE\bigl[\|X-\Xh_n\|_1 \bigr] \\
	&\leq \inf_{\Xh_n\in\Ac_n^\mathrm{eq}(L_1([0,T]),X(0),W)}
		\EE\bigl[\|X-\Xh_n\|_{1} \bigr] \leq C\cdot n^{-1/2}.
\end{align*}
In particular, adaptive algorithms are not superior (up to multiplicative constants)
to algorithms that are based on fixed equidistant grids for global approximation with respect to the $L_1$-norm.

A further example of an SDE with discontinuous coefficients that is often considered
in the literature is given by
\begin{align*}
	\dd X(t) = \sgn(X(t))\,\dd t + \dd W(t), \qquad X(0)=x_0,
\end{align*}
with initial value $x_0\in\R$, see, e.g., \citet*[Table~1]{GLN17}.
Observe that in this case the assumptions of Theorem~\ref{thm:ft-local} for pointwise approximation
are not fulfilled since $a'b - ab' - \tfrac{1}{2} b^2b''=0$ on $\R\setminus\{0\}$. As a consequence, it could be possible that adaptive algorithms are able to achieve
a convergence rate greater than $1$ for this SDE.

\appendix
\section{Properties of the Normal Distribution}

We collect a number of properties of the Gaussian distribution on the real line,
which are employed in the proofs of Proposition~\ref{prop:ft-global},
Proposition~\ref{prop:sup-global}, and Proposition~\ref{prop:Lp-global}.
We suppose that these properties are well-known but for the convenience of the reader
we provide proofs of these facts.

\begin{lem}\label{lem:normal1}
	Let $Z$ be a real-valued random variable that is normally distributed.
	Then for all $\sigma\in [0,\infty)$ with $\sigma^2\leq \Var[Z]$ and $\varepsilon\in[0,\infty)$ we have
	\begin{align*}
		\PP(|Z| \geq \varepsilon\,\sigma) \geq 1- \varepsilon.
	\end{align*}
\end{lem}

\begin{proof}
	We may assume $\Var[Z]>0$ and $\sigma^2=\Var[Z]$.
	Let $\varepsilon\in[0,\infty)$ and $\mu=\EE\left[Z\right]$.
	By the Anderson inequality we have $\PP(|Z|<\varepsilon\,\sigma)\leq \PP(|Z-\mu|<\varepsilon\,\sigma)$,
	see \citet*{And55} or \citet*[Corollary~7.1, p.~47]{Lif12}. Hence
	\begin{align*}
		\PP(|Z| \geq \varepsilon\,\sigma)
			\geq 1 - 2\int_0^\varepsilon \frac{1}{\sqrt{2\pi}} \exp(-x^2/2)\,\dd x
			\geq 1- \frac{2}{\sqrt{2\pi}}\,\varepsilon \geq 1-\varepsilon.
	\end{align*}
\end{proof}

\begin{lem}\label{lem:normal2}
	Let $\varepsilon,\delta\in (0,\infty)$.
	Then there exists a constant $c\in (0,\infty)$ with the following property.
	If $N\in\N$ and $Z_1,\dots,Z_N$ are independent real-valued random variables
	each being normally distributed with variance at least $\delta^2$, then
	\begin{enumerate}[(i)]\addtolength{\itemsep}{0.25\baselineskip}
	\item\label{lem18parti} $\PP\Bigl(\sum_{i=1}^N |Z_i| \geq c\,N \Bigr) \geq 1-\varepsilon$,
	\item\label{lem18partii} $\PP\Bigl(\max_{i=1,\dots,N} |Z_i| \geq c\,\sqrt{\ln(N)} \Bigr) \geq 1-\varepsilon$.
	\end{enumerate}
\end{lem}

\begin{proof}
	Let $\varepsilon,\delta\in (0,\infty)$ and let $(Y_i)_{i\in\N}$ be a sequence of
	independent, real-valued, standard normal random variables.
	
	Put $M_N = N^{-1}\cdot\sum_{i=1}^N |Y_i|$ for $N\in\N$. By the strong law of large numbers
	we have $\lim_{N\to \infty} M_N = \EE[|Y_1|] = \sqrt{2/\pi}$ with probability one.
	Since $\sqrt{2/\pi} > 1/2$ we conclude that $\lim_{N\to\infty}\PP(M_N \le 1/2) = 0$. Hence 
	\begin{align}\label{a111}
		\exists\,N_1\in\N \enspace \forall\,N\geq N_1 \colon\, \PP\Bigl( \sum_{i=1}^N |Y_i| \leq N/2 \Bigr) \leq \varepsilon.
	\end{align}
	Furthermore, since $\PP(\min_{N=1,\dots, N_1-1} M_N >0) = 1$, there exists $c_1\in (0,\infty)$ such that
	\begin{align}\label{a112}
		\forall\,N < N_1 \colon\, \PP\Bigl(\sum_{i=1}^N |Y_i| \leq c_1 N \Bigr) \leq \varepsilon.
	\end{align}
	
	Next, define $\varphi\colon (1,\infty) \to (0,\infty)$ by
	\[
		\varphi(t) = \frac{1}{\sqrt{2\pi}}\cdot {\left(\frac{1}{t}-\frac{1}{t^3}\right)} \cdot\exp(-t^2/2).
	\]
	Using a well-known bound for the tails of the standard normal distribution,
	see, e.g., \citet*[Lemma~1.1.3]{Bog98}, we have
	\begin{align*}
		\forall\,N\in\N,\ t\in(1,\infty)\colon\, \PP\Bigl(\max_{i=1,\dots,N}|Y_i|\leq t\Bigr)
			= \PP\bigl(|Y_1|\leq t\bigr)^N \leq (1-2\varphi(t))^N \leq \exp\bigl(-2\varphi(t)N\bigr).
	\end{align*}
	Note that $\lim_{N\to\infty}\varphi\bigl(c\,\sqrt{\ln(N)}\bigr)\cdot N=\infty$ for every $c\in (0,\sqrt{2})$. Hence 
	\begin{align}\label{a115}
		\exists\,N_2\geq 2 \enspace \forall\,N\geq N_2 \colon\,
			\PP\Bigl(\max_{i=1,\dots,N} |Y_i| \leq \sqrt{\ln(N)}\Bigr) \leq \varepsilon.
	\end{align}
	Put $c_2 = \varepsilon /\sqrt{\ln (N_2)}\in(0,\infty)$. Employing Lemma~\ref{lem:normal1} we obtain 
	\begin{align}\label{a213}
		\forall\, N < N_2\colon\, \PP\Bigl( \max_{i=1,\dots,N}|Y_i| \leq c_2\sqrt{\ln(N)} \Bigr)
			\leq \PP\bigl( |Y_1| \leq c_2\sqrt{\ln(N_2)} \bigr) \leq \varepsilon.
	\end{align}
	
	Finally, let $(\mu_i)_{i\in\N}\subset\R$, $(\sigma_i)_{i\in\N}\subset [\delta,\infty)$
	and let $(Z_i)_{i\in\N}$ be a sequence of independent, real-valued random variables
	with $Z_i\sim\mathrm{N}(\mu_i,\sigma_i^2)$ for every $n\in\N$. By the Anderson inequality,
	see \citet*{And55} or \citet*[Corollary~7.1, p.~47]{Lif12}, and the assumption on the variances $(\sigma_i)_{i\in\N}$
	we have for every $p\in[1,\infty]$ that
	\begin{align}\label{a123}
		\begin{aligned}
			\forall\,N\in\N,\ z\in[0,\infty)\colon\,
				\PP\bigl( |(Z_i)_{i=1,\dots,N}|_p \leq z \bigr) &\leq \PP\bigl( |(Z_i-\mu_i)_{i=1,\dots,N}|_p \leq z \bigr) \\
				&\leq \PP\Bigl( \bigl|\bigl((Z_i-\mu_i)/\sigma_i\bigr)_{i=1,\dots,N}\bigr|_p \leq z/\delta \Bigr).
		\end{aligned}
	\end{align}
	Let $c=\min(1/2,c_1,c_2)\cdot \delta\in(0,\infty)$.
	Combining~\eqref{a111},~\eqref{a112}, and~\eqref{a123} for $p=1$ we obtain
	\[
	\forall\,N\in\N \colon\, \PP\Bigl( \sum_{i=1}^N |Z_i| \leq c\,N \Bigr)
		\leq \PP\Bigl( \sum_{i=1}^N |Y_i| \leq \min(1/2,c_1)\cdot N \Bigr) \leq \varepsilon,
	\]
	which yields the statement in~\eqref{lem18parti}.
	Combining~\eqref{a115},~\eqref{a213}, and~\eqref{a123} for $p=\infty$ we obtain
	\[
	\forall\,N\in\N\colon\, \PP\Bigl( \max_{i=1,\dots,N} |Z_i| \leq c\,\sqrt{\ln(N)} \Bigr)
		\leq \PP\Bigl( \max_{i=1,\dots,N} |Y_i| \leq \min(1,c_2)\,\sqrt{\ln(N)} \Bigr) \leq \varepsilon,
	\] 
	which yields the statement in \eqref{lem18partii} and completes the proof.
\end{proof}

\begin{lem}\label{lem:normal3}
	Let $\varepsilon,\delta\in (0,\infty)$. 
	Then there exists a constant $c\in (0,\infty)$ with the following property.
	If $k\in2\N$, $B_1,\dots,B_{k/2}$ are independent Brownian bridges (from $0$ to $0$) on $[0,1/k]$,
	$f_1,\dots,f_{k/2}\in L_1([0,1/k])$, and $a_1,\dots,a_{k/2}\in\R$ with $|a_1|,\dots,|a_{k/2}|\geq\delta$, then
	\begin{align*}
		\PP\Bigl( \sum_{i=1}^{k/2} \|a_i\cdot B_i-f_i\|_1 \geq c/\sqrt{k/2} \Bigr) \geq 1-\varepsilon.
	\end{align*}
\end{lem}

\begin{proof}
Let $\varepsilon,\delta\in (0,\infty)$ and choose $c\in (0,\infty)$ according to Lemma~\ref{lem:normal2}.
	For $i=1,\dots,k/2$ define
	\[
	Z_i = \sqrt{12}\, k^{3/2}\cdot \int_0^{1/k} \bigl( a_i\cdot B_i(t)-f_i(t) \bigr)\,\dd t.
	\]
	Then $Z_1,\dots,Z_{k/2}$ are independent normal random variables with 
	\begin{align*}
		\Var(Z_i) = 12 k^3 a_i^2 \cdot \EE\biggl[\Bigl( \int_0^{1/k} B_i(t)\,\dd t \Bigr)^2\biggr] = a_i^2 \geq \delta^2
	\end{align*}
	for $i=1,\dots,k/2$, see, e.g., \citet*[Equation~(16)]{MG04}.
	Thus, by Lemma~\ref{lem:normal2}(\ref{lem18parti}),
	\[
	\PP\Bigl( \sum_{i=1}^{k/2} \|a_i\cdot B_i-f_i\|_1 \geq c/\sqrt{48k} \Bigr)
		\geq \PP\Bigl( \sum_{i=1}^{k/2} |Z_i| \geq  c\cdot k/2 \Bigr) \geq 1-\varepsilon,
	\]
	which completes the proof.
\end{proof}

\section{Localization Technique: Comparison Results}\label{sec:comparison-result}

We relate the solution $X$ of the integral equation \eqref{eq:SDE} in Section~\ref{sec:setting} to the
solution $\Xt$ of an integral equation with the same initial value $\Xt(0)=X(0)$ and
coefficients $\at$ and $\bt$ that coincide with $a$ and $b$
on a stripe $[0,T]\times I\subseteq [0,T]\times \R$, respectively.

\begin{lem}\label{lem:localization}
	Assume the setting in Section~\ref{sec:setting}. Let $\at\colon [0,T]\times \R\to\R$
	and $\bt\colon [0,T]\times\R\to\R$ be Borel-measurable functions, let $\emptyset\neq I\subseteq\R$ be an
	open interval and let $C\in [0,\infty)$ be a constant such that
	for all $t\in [0,T]$ and $x\in\R$ it holds that
	\begin{align*}
		|\at(t,x)| + |\bt(t,x)| \leq C\, (1+|x|),
	\end{align*}
	for all $t\in [0,T]$ and $x,y\in I$ it holds that
	\begin{align}\label{eq:p1}
		|\at(t,x)-\at(t,y)| + |\bt(t,x)-\bt(t,y)| \leq C\, |x-y|,
	\end{align}
	and for all $t\in [0,T]$ and $x\in I$ it holds that
	\begin{align}\label{eq:p2}
		a(t,x)=\at(t,x),\qquad b(t,x)=\bt(t,x).
	\end{align}
	Assume further that $\Xt\colon [0,T]\times\Omega\to\R$ is an $(\Fc_t)_{t\in [0,T]}$-adapted
	stochastic process with continuous paths such that
	for all $t\in[0,T]$ it holds $\PP$-a.s.~that
	\begin{align}\label{eq:Xt}
		\Xt(t) = X(0) + \int_0^t \at\bigl(s,\Xt(s)\bigr)\,\dd s + \int_0^t \bt\bigl(s,\Xt(s)\bigr)\,\dd W(s).
	\end{align}
	Define the stopping time
	\begin{align}\label{eq:stop}
		\tau = \inf\{t\in[0,T]\colon X(t) \not\in I\} \wedge\inf\{t\in[0,T]\colon \Xt(t)\not\in I\} \wedge T.
	\end{align}
	Then $\PP$-a.s.~for every $t\in[0,T]$
	\begin{align*}
		X({t\wedge\tau}) = \Xt({t\wedge\tau}).
	\end{align*}
	Moreover,
	\begin{align*}
		\PP\bigl( \{ \forall\,t\in[0,T]\colon \Xt(t)=X(t)\} \cap \{ \forall\,t\in[0,T]\colon \Xt(t)\in I\} \bigr)
			= \PP\bigl( \forall\,t\in[0,T]\colon \Xt(t)\in I \bigr).
	\end{align*}
\end{lem}

\begin{proof}
	For every $n\in\N$ we define a stopping time by
	\[
	\sigma_n = \inf\{t\in [0,T]\colon \max(|X(t)|,|\Xt(t)|)\geq n\}\wedge T.
	\] 
	Then  
	\begin{align}\label{eq:e1}
		X(t\wedge \tau) = \lim_{n\to\infty} X(t\wedge \tau\wedge \sigma_n),\qquad
			\Xt(t\wedge \tau) = \lim_{n\to\infty} \Xt(t\wedge \tau\wedge \sigma_n)
	\end{align}
	for every $t\in [0,T]$. In view of \eqref{eq:e1} and the pathwise continuity of $X$ and $\Xt$
	it remains to prove that for every $t\in[0,T]$ and every $n\in\N$ it holds $\PP$-a.s.~
	\[
	X(t\wedge \tau\wedge \sigma_n) = \Xt(t\wedge \tau\wedge \sigma_n).
	\]  

	To this end we fix $n\in\N$ and put $\rho = \tau\wedge \sigma_n$. Moreover, we put 
	\[
	A(t)= a(t,X(t)),\qquad \At(t) = \at\bigl(t,\Xt(t)\bigr),\qquad \Ab(t) = \at(t,X(t))
	\]
	as well as 
	\[
	B(t)= b(t,X(t)),\qquad \Bt(t) = \bt\bigl(t,\Xt(t)\bigr),\qquad \Bb(t) = \bt(t,X(t))
	\]
	for every $t\in[0,T]$.
	By \eqref{eq:p2} we have 
	\[
	\Ab = A\text{ and } \Bb = B\text{ on }\{(t,\omega)\in [0,T]\times\Omega\colon 0\leq t< \rho(\omega)\}.
	\]
	It follows that for all $t\in[0,T]$ we have $\PP$-a.s.~ 
	\begin{align*}
		X(t\wedge \rho) &= X(0) + \int_0^{t\wedge \rho} A(s)\,\dd s + \int_0^{t\wedge \rho} B(s)\,\dd W(s) \\
		&= X(0) + \int_0^{t\wedge \rho} \Ab(s)\,\dd s + \int_0^{t\wedge \rho} \Bb(s)\,\dd W(s),
	\end{align*}
	and therefore $\PP$-a.s.~
	\begin{align}\label{eq:e2}
		X(t\wedge \rho) - \Xt(t\wedge \rho)
			= \int_0^{t\wedge \rho} \bigl(\Ab(s)- \At(s)\bigr)\,\dd s + \int_0^{t\wedge \rho} \bigl(\Bb(s)-\Bt(s)\bigr)\,\dd W(s).
	\end{align}

	By the H\"older inequality and \eqref{eq:p1} we get
	\begin{align}\label{eq:e3}
		\begin{aligned}
			\EE\biggl[\Bigl(\int_0^{t\wedge \rho} \bigl(\Ab(s)- \At(s)\bigr)\,\dd s\Bigr)^2\biggr]
				&\leq  T\,\EE\biggl[\int_0^{t\wedge \rho} \bigl|\Ab(s)- \At(s)\bigr|^2\,\dd s\biggr] \\
			&\leq  TC^2\int_0^t \EE\Bigl[ \bigl|X(s\wedge \rho)-\Xt(s\wedge \rho)\bigl|^2\Bigl]\,\dd s.
		\end{aligned}
	\end{align}
	By the It\^{o} isometry and \eqref{eq:p1} we get
	\begin{align}\label{eq:e4}
		\begin{aligned}
			\EE\biggl[\Bigl(\int_0^{t\wedge \rho} \bigl(\Bb(s)- \Bt(s)\bigr)\,\dd W(s)\Bigr)^2\biggr]
				&= \EE\biggl[ \int_0^{t\wedge\rho} \bigl|\Bb(s)- \Bt(s)\bigr|^2\,\dd s\biggr] \\
			&\leq C^2 \int_0^t \EE\Bigl[ \bigl|X(s\wedge\rho)-\Xt(s\wedge\rho)\bigr|^2 \Bigr]\,\dd s.
		\end{aligned}
	\end{align}

	Define a bounded and Borel-measurable function $f\colon[0,T]\to[0,\infty)$ by $f(t)=\EE[|X(t\wedge \rho)-\Xt(t\wedge\rho)|^2]$.
	Using \eqref{eq:e2}, \eqref{eq:e3}, and \eqref{eq:e4} we conclude that for all $t\in[0,T]$,
	\[
	f(t) \leq 2\,(TC^2+C^2)\int_0^t f(s)\,\dd s.
	\] 
	Hence $f=0$ by the Gronwall inequality. This proves the first part of the lemma.

	For the second part, let $f,g\colon [0,T]\to\R$ be continuous functions, let
	\begin{align*}
		\tau_0 = \inf\{t\in[0,T]\colon f(t) \not\in I\} \wedge\inf\{t\in[0,T]\colon g(t) \not\in I\} \wedge T,
	\end{align*}
	and assume that for all $t\in[0,T]$
	\begin{align*}
		f(t\wedge\tau_0) = g(t\wedge \tau_0) \quad\text{and}\quad f(t)\in I.
	\end{align*}
	If $\tau_0<T$, then the continuity of $g$ and openness of $I$ imply $g(\tau_0)\not\in I$
	and hence $f(\tau_0)=g(\tau_0)\not\in I$, which is a contradiction.
	This shows $\tau_0=T$ and hence $f=g$.
\end{proof}

Next, we provide sufficient conditions for an It\^{o} process to stay (over time)
in an open interval with positive probability.
The following lemma is a consequence of the Girsanov theorem and the method of time change.

\begin{lem}\label{lem:time-change}
	Assume the setting in Section~\ref{sec:setting}.
	Let $\alpha\colon[0,\infty)\times\Omega\to\R$ and $\beta\colon[0,\infty)\times\Omega\to\R$ be bounded
	processes that are progressively measurable with respect to $(\Fc_{t})_{t\in [0,\infty)}$
	such that $\beta$ is bounded away from zero.
	Let $I\subseteq\R$ be an open interval such that $\PP(X(0)\in I)>0$.
	Then we have for all $S\in(0,\infty)$ that
	\[
	\PP\Bigl(\forall\,t\in[0,S]\colon \Big[X(0)+\int_0^t\alpha(s)\,\dd s + \int_0^t\beta(s)\,\dd W(s)\Big] \in I\Bigr) > 0.
	\]
\end{lem}

\begin{proof}
By the properties of $\beta$, the process $Z\colon[0,\infty)\times\Omega\to [0,\infty)$ defined by
	\begin{align*}
		Z(t) = \int_0^t \beta^2(s)\,\dd s
	\end{align*}
	is strictly increasing
	as well as (Lipschitz) continuous and satisfies $Z(0)=0$ as well as $\lim_{t\to\infty}Z(t)=\infty$.
	Let $Z^{-1}$ denote the inverse of $Z$. Then  $Z^{-1}(t)$ is a stopping time for every $t\in [0,\infty)$ and by the Dambis/Dubins-Schwarz theorem, see, e.g., \citet*[Theorem~3.4.6]{KS91}, we get that the process $B\colon[0,\infty)\times\Omega\to\R$ defined by
	\begin{align}\label{eq:def-B}
		B(t)=\int_0^{Z^{-1}(t)} \beta(s)\,\dd W(s)
	\end{align}
	is a standard Brownian motion with respect to the normal filtration $(\Gc_t)_{t\in [0,\infty)}=(\Fc_{Z^{-1}(t)})_{t\in [0,\infty)}$.
	
Let $\omega\in\Omega$. Then $Z(\omega,\cdot)\colon [0,\infty) \to [0,\infty)$ is absolutely continuous with $Z(\omega,\cdot)'=\beta^2(\omega,\cdot) >0$ Lebesgue almost everywhere.	It follows that the inverse $Z^{-1}(\omega,\cdot)\colon [0,\infty)\to [0,\infty)$ is absolutely continuous with $(Z^{-1}(\omega,\cdot))' = 1/(\beta^2(\omega,\cdot)\circ Z^{-1}(\omega,\cdot))$
Lebesgue almost everywhere, see, e.g., \citet*[Exercises~5.8.51,~5.8.52]{Bog2007}. Hence, by the change of variable formula we have 
	for every $t\in[0,\infty)$ that
	\begin{align}\label{eq:substit}
		\int_0^{t} \alpha(\omega,s)\,\dd s
			= \int_0^{Z(t)} \frac{\alpha\bigl(\omega,Z^{-1}(\omega,s)\bigr)}{\beta^2\bigl(\omega,Z^{-1}(\omega,s)\bigr)}\,\dd s.
	\end{align}
Combining  \eqref{eq:def-B} and \eqref{eq:substit} yields that for every $t\in[0,\infty)$ we have
	\begin{align}\label{eq:time-change}
		X(0) + \int_0^{Z^{-1}(t)}\alpha(s)\,\dd s + \int_0^{Z^{-1}(t)}\beta(s)\,\dd W(s)
			= X(0) + \int_0^t \frac{\alpha\bigl(Z^{-1}(s)\bigr)}{\beta^2\bigl(Z^{-1}(s)\bigr)}\,\dd s + B(t).
	\end{align}
Since the processes $\alpha$ and $\beta$ are progressively measurable with respect to $(\Fc_t)_{t\in [0,\infty)}$ and the process $Z^{-1}$ is measurable, we get that the process $\gamma=\bigl(-\alpha(Z^{-1}(s))/\beta^2(Z^{-1}(s))\bigr)_{s\in [0,\infty)}$
	is measurable and adapted with respect to $(\Gc_t)_{t\in [0,\infty)}$. Moreover, by the boundedness properties of $\alpha$ and $\beta$, the process $\gamma$ 	satisfies the Novikov condition, 	see, e.g., \citet*[Corollary~3.5.13]{KS91}. Fix $S\in (0,\infty)$.
	Then, by the Girsanov theorem, see, e.g., \citet*[Theorem~3.5.1]{KS91}, the mapping
	\[
	\widetilde\PP\colon \Gc_{S}\to [0,1], \quad
		A\mapsto \EE\biggl[ \ind_A\cdot \exp\Bigl(
		\int_0^{S}\gamma(t)\,\dd B(t)-\frac{1}{2}\int_0^{S}\gamma^2(t)\,\dd t\Bigr)\biggr]
	\] 
	is a probability measure on $\Gc_{S}$, which is equivalent to $\PP\vert_{\Gc_{S}}$,
	and the process $\Bt\colon[0,S]\times\Omega\to\R$ given by 
	\[
	\Bt(t) = B(t)- \int_0^t \gamma(s)\,\dd s
	\]
	is a $(\Gc_t)_{t\in[0,S]}$-adapted Brownian motion on $[0,S]$ with respect to $\widetilde\PP$.
	By~\eqref{eq:time-change} we then obtain that
	\begin{align*}
	0&<\widetilde\PP\Bigl(\forall\,t\in[0,S]\colon \Big[X(0)+\Bt(t) \Big] \in I\Bigr) \\
		& = \widetilde\PP\Bigl(\forall\,t\in[0,S]\colon \Big[X(0)+\int_0^{Z^{-1}(t)}\alpha(s)\,\dd s + \int_0^{Z^{-1}(t)}\beta(s)\,\dd W(s) \Big]\in I \Bigr).
	\end{align*}
	Hence it holds that
		\begin{align*}
		0&<
\PP\Bigl(\forall\,t\in[0,S]\colon \Big[X(0)+\int_0^{Z^{-1}(t)}\alpha(s)\,\dd s + \int_0^{Z^{-1}(t)}\beta(s)\,\dd W(s) \Big]\in I \Bigr).
	\end{align*}
	Combining this with the fact that there exists $c\in (0,\infty)$ such that $Z^{-1}(t)\ge c\,t$ for every $t\in [0,\infty)$ completes the proof of the lemma.
\end{proof}

The following proposition establishes a comparison result for SDEs.
It is a key tool to relax global assumptions on the coefficients in order to obtain lower error bounds for SDEs under
non-standard assumptions on the coefficients.

\begin{prop}[Comparison result]\label{prop:localization}
	Assume the setting in Section~\ref{sec:setting}. Let $\at\colon[0,T]\times\R\to\R$
	and $\bt\colon[0,T]\times\R\to\R$ be bounded and Borel-measurable functions such that
	\begin{align}\label{eq:ass1}
		\inf_{(t,x)\in [0,T]\times \R} |\bt(t,x)| > 0
	\end{align}
	and let $C\in [0,\infty)$ be a constant such that for all $t\in[0,T]$ and $x,y\in\R$ it holds that
	\begin{align*}
		|\at(t,x)-\at(t,y)| + |\bt(t,x)-\bt(t,y)| \leq C\,|x-y|.
	\end{align*}
	Moreover, let $I\subseteq\R$ be an open interval such that
	\begin{align*}
		\PP\big(X(0)\in I\big)>0
	\end{align*}
	and for all $t\in [0,T]$ and $x\in I$ it holds that
	\begin{align*}
		a(t,x)=\at(t,x),\qquad b(t,x)=\bt(t,x).
	\end{align*}
	Assume further that $\Xb\colon [0,T]\times\Omega\to\R$ is an $(\Fc_t)_{t\in [0,T]}$-adapted stochastic process
	with continuous sample paths such that
	\begin{align}\label{eq:ass2}
		\PP\big( \{X(0)=\Xb(0)\} \cap \{X(0)\in I\} ) = \PP\big(X(0)\in I\big)
	\end{align}
	and for all $t\in[0,T]$ it holds $\PP$-a.s.~that 
	\begin{align*}
		\Xb(t) = \Xb(0) + \int_0^t \at\bigl(s,\Xb(s)\bigr)\,\dd s + \int_0^t \bt\bigl(s,\Xb(s)\bigr)\,\dd W(s).
	\end{align*}
	Then we have
	\begin{align*}
		\PP\big(\forall\,t\in[0,T]\colon \Xb(t)=X(t)\big)>0.
	\end{align*}
\end{prop}

\begin{proof}
	Let $\Xt\colon[0,T]\times\Omega\to\R$ be an $(\Fc_t)_{t\in [0,T]}$-adapted stochastic process
	with continuous sample paths such that for all $t\in [0,T]$ it holds $\PP$-a.s.~that 
	\begin{align*}
		\Xt(t) = X(0) + \int_0^t \at\bigl(s,\Xt(s)\bigr)\,\dd s + \int_0^t \bt\bigl(s,\Xt(s)\bigr)\,\dd W(s).
	\end{align*}
	Define $\ah\colon[0,\infty)\times\R\to\R$ and $\bh\colon[0,\infty)\times\R\to\R$ by
	\begin{align*}
		\ah(t,x) =
		\begin{cases}
			\at(t,x),  &\text{if }t\in [0,T],\\
			0, &\text{else},
		\end{cases}\qquad
		\bh(t,x) =
		\begin{cases}
			\bt(t,x), &\text{if }t\in [0,T],\\
			1, &\text{else}.
		\end{cases}
	\end{align*}
	Observe that $\ah$ and $\bh$ are Borel-measurable and bounded. Moreover, \eqref{eq:ass1} yields
	\begin{align*}
		\inf_{(t,x)\in [0,\infty)\times \R} |\bh(t,x)| > 0.
	\end{align*}
	Define the process $\Xhat\colon[0,\infty)\times\Omega\to\R$ by
	\begin{align*}
		\Xhat(t) =
		\begin{cases}
			\Xt(t), &\text{if }t\in [0,T],\\
			W(t)-W(T)+\Xt(T), &\text{else}.
		\end{cases}
	\end{align*}
	Observe that $\Xhat$ is an $(\Fc_t)_{t\in [0,\infty)}$-adapted stochastic process
	with continuous sample paths and for all $t\in[0,\infty)$ it holds $\PP$-a.s.~that 
	\begin{align*}
		\Xhat(t) = X(0) + \int_0^t \ah\bigl(s,\Xhat(s)\bigr)\,\dd s + \int_0^t \bh\bigl(s,\Xhat(s)\bigr)\,\dd W(s).
	\end{align*}
	We may thus apply Lemma~\ref{lem:time-change} with $\alpha(t)=\ah(t,\Xhat(t))$, $\beta(t)=\bh(t,\Xhat(t))$ and $S=T$ to obtain
	\begin{align*}
		\PP\bigl( \forall\,t\in[0,T]\colon \Xhat(t)\in I \bigr) > 0,
	\end{align*}
	and hence
	\begin{align}\label{eq:5}
		\PP\bigl( \forall\,t\in[0,T]\colon \Xt(t)\in I \bigr) > 0.
	\end{align}
	Combining \eqref{eq:ass2} with the facts $\PP\bigl(\Xt(0)=X(0)\bigr)=1$ and
	\begin{align*}
		\PP\big( \forall\,t\in[0,T]\colon \Xt(t)=\Xb(t) \big) = \PP\big( \Xt(0)=\Xb(0) \big)
	\end{align*}
	yields
	\begin{align}\label{eq:6}
		\PP\bigl( \{\forall\,t\in[0,T]\colon \Xt(t)=\Xb(t)\} \cap \{\forall\,t\in[0,T]\colon \Xt(t)\in I\} \bigr)
			= \PP\bigl( \forall\,t\in[0,T]\colon \Xt(t)\in I \bigr).
	\end{align}
	Lemma~\ref{lem:localization} shows that
	\begin{align*}
		\PP\bigl( \{\forall\,t\in[0,T]\colon \Xt(t)=X(t)\} \cap \{\forall\,t\in[0,T]\colon \Xt(t)\in I\} \bigr)
			= \PP\bigl( \forall\,t\in[0,T]\colon \Xt(t)\in I \bigr).
	\end{align*}
	Combining this with \eqref{eq:6} yields
	\begin{align*}
		\PP\bigl( \{\forall\,t\in[0,T]\colon \Xb(t)=X(t)\} \cap \{\forall\,t\in[0,T]\colon \Xt(t)\in I\} \bigr)
			= \PP\bigl( \forall\,t\in[0,T]\colon \Xt(t)\in I \bigr).
	\end{align*}
	Combining this with \eqref{eq:5} yields the claim.
\end{proof}

\section*{Acknowledgment}

Mario Hefter is supported by the Austrian Science Fund (FWF), Project F5506-N26,
which is part of the Special Research Program
``Quasi-Monte Carlo Methods: Theory and Applications''.

\bibliographystyle{plainnat}
\bibliography{bib}

\end{document}